\documentclass[11pt, oneside,reqno]{amsart}

\usepackage{amsmath,amssymb,amsthm,textcomp,mathtools}
\usepackage{amsfonts,graphicx}
\usepackage[mathscr]{eucal}
\usepackage{color}
\usepackage{csquotes}
\usepackage{enumitem}
\usepackage{setspace}
\numberwithin{equation}{section}

\theoremstyle{definition}

\numberwithin{equation}{section}

\date{\today}

\newtheorem{theorem}{\bf Theorem}[section]
\newtheorem{remark}{\bf Remark}[section]
\newtheorem{proposition}{Proposition}[section]

\newtheorem{definition}{Definition}[section]
\newtheoremstyle
{remarkstyle}
{}
{11pt}
{}
{}
{\bfseries}
{:}
{     }
{\thmname{#1} \thmnumber{#2} }

\theoremstyle{remarkstyle}

\begin{document}
	\title{Some Spatial Point Processes of Poisson Family}
	 \author[Pradeep Vishwakarma]{Pradeep Vishwakarma}
	 \address{Pradeep Vishwakarma, Theoretical Statistics and Mathematics Units,
	 	Indian Statistical Institute, Kolkata, 700108, West Bengal, India.}
	 \email{vishwakarmapr.rs@gmail.com}

	\subjclass[2010]{Primary : 60G55; Secondary: 60G57, 60G60}
	
	\keywords{Poisson random field, generalized Poisson random field, inverse stable subordinator, fractional generalized Poisson random field, thinning, Caputo fractional derivative, generalized Skellam point process}
	\date{\today}
	
	\maketitle
	\begin{abstract}
	Spatial Poisson point processes on finite-dimensional Euclidean space provide fundamental mathematical tools for modeling random spatial point patterns. In this paper, we introduce and analyze several Poisson-type spatial point processes. In particular, we propose and study a point process, namely, the generalized Poisson random field (GPRF), in which more than one point can be observed with positive probability, within a rectangular region having infinitesimal Lebesgue measure. A thinning of the GPRF into independent GPRFs with reduced rate parameters is discussed. Furthermore, we consider these processes indexed by the positive quadrant of the plane and analyze their fractional variants. Various distributional properties of these processes and related governing differential equations are obtained. Later, we define and analyze a spatial Skellam-type point process via GPRF. Moreover, a fractional variant of it in the two parameter case is studied in detail.
	\end{abstract}
	
\section{Introduction}\label{sec1}
 A spatial Poisson point process on Euclidean space models the random distribution of points in a finite-dimensional domain. It is a fundamental stochastic model for analyzing spatial point patterns, where points represent the locations of different types of objects or events.   Let  $\mathbb{R} _+ = [0,\infty)$ denote the set of non-negative real numbers. For $d\ge1$, let $\mathcal{A}_d=\{[\textbf{0},\textbf{x}],\ \textbf{x}\in\mathbb{R}^d_+\}$ denote the collection of rectangles in $\mathbb{R}^d_+$, where $[\textbf{0},\textbf{x}]=\prod_{j=1}^{d}[0,x_j]$ for $\textbf{x}=(x_1,\dots,x_d)\in\mathbb{R}^2_+$. For $A\in\mathcal{A}_d$, we denote its Lebesgue measure by $|A|$.  The spatial Poisson point process (or Poisson random field) is a non-negative integer-valued random measure $\{N(A),\ A\in\mathcal{A}_d\}$, which is characterized as follows:\\
  \noindent (i) There exist a constant $\lambda>0$ such that for any $A\in\mathcal{A}_d$ with $|A|<\infty$, the random points count $N(A)$ in set $A$ has Poisson distribution with mean $\lambda|A|$;\\
  \noindent (ii) for any finite collection $\{A_1,A_2,\dots,A_m\}\subset\mathcal{A}_d$ of disjoint  sets, the random variables $N(A_1)$, $N(A_2)$, $\dots,N(A_m)$ are independent  of each other.\\  
  We call $\{N(A),\ A\in\mathcal{A}_d\}$ the Poisson random field (PRF). Its infinitesimal probabilities characterization is as follows (see \cite{Stoyen1995}):
  Let $A\in\mathcal{A}_d$ be a rectangle with infinitesimal Lebesgue measure, that is, $o(|A|)/|A|\to0$ as $|A|\to0$. Then, the probability of no point in $A$ is $\mathrm{Pr}\{N(A)=0\}=1-\lambda|A|+o(|A|)$, the probability of exactly one point in $A$ is $\mathrm{Pr}\{N(A)=1\}=\lambda|A|+o(|A|)$ and the probability of more than one point in $A$ is $\mathrm{Pr}\{N(A)>1\}=o(|A|)$. For $d=2$, the PRF can be defined as a non-negative integer valued two parameter L\'evy process (for definition, see Section \ref{gprf2sec}) $\{N(s,t), (s,t)\in\mathbb{R}^2_+\}$ on positive quadrant of the plane, with rectangular increments. Here,
  $ N(s,t)\coloneqq N([0,s]$ $\times[0,t])$, $(s,t)\in\mathbb{R}_+^2
  $ for all rectangle $[0,s]\times[0,t]\subset\mathbb{R}^2_+$. 
  
  Two-parameter point processes with independent rectangular increments constitute a fundamental class of planar stochastic models. A systematic study of their characterization under random time changes is of significant theoretical and practical interest. The incorporation of random time changes enhances the real life applications of these processes, enabling the representation of irregular, non-stationary behavior and complex dependence structures that cannot be adequately captured by their classical counterparts. Therefore, the time-changed point processes on $\mathbb{R}^2_+$ will provide a general and advanced mathematical framework for modeling spatial stochastic systems exhibiting long-range dependence and complex correlation structures. Such models arise naturally in various applied fields, such as statistical signal processing and financial market analysis. In financial applications, these processes are widely used to model asset returns, where the stochastic time-change mechanism is commonly interpreted as a random economic or business clock governing market activity. 
  
  Over the past few years, time-changed two-parameter Poisson processes have been extensively studied, as they provide a flexible framework for modeling complex spatial systems exhibiting long range dependence structure. For example, in \cite{Leonenko2015}, a time-changed version of the PRF on $\mathbb{R}^2_+$, driven by an independent bivariate inverse subordinator is introduced and studied. Further, various other characterization of this process are obtained in \cite{Aletti2018}. Recently, in \cite{Kataria2024}, a time fractional variants of the PRF, $\{N^{\alpha,\beta}(s,t),\ (s,t)\in\mathbb{R}^2_+\}$, $0<\alpha,\beta\leq1$ is introduced and studied. It is defined as a process whose one dimensional distribution $q^{\alpha,\beta}(n,s,t)=\mathrm{Pr}\{N^{\alpha,\beta}(s,t)=n\}$, $n\ge0$ solves the following system of fractional differential equations:
  \begin{equation*}
  	\frac{\partial^{\alpha+\beta}}{\partial t^{\beta}\partial s^{\alpha}}q^{\alpha,\beta}(n,s,t)=(n+1)\lambda q^{\alpha,\beta}(n+1,s,t)-(2n+1)\lambda q^{\alpha,\beta}(n,s,t)+n\lambda q^{\alpha,\beta}(n-1,s,t),
  \end{equation*}
  with initial conditions $q^{\alpha,\beta}(0,0,0)=q^{\alpha,\beta}(0,s,0)=q^{\alpha,\beta}(0,0,t)=1$ for all $s\ge0$ and $t\ge0$, where the partial derivatives are in the sense of Caputo fractional derivative, defined as (see \cite{Kilbas2006})
  \begin{equation}\label{caputoder}
  	\frac{\mathrm{d}^\alpha}{\mathrm{d}t^\alpha}f(t)\coloneqq\begin{cases}
  	 \frac{1}{\Gamma(1-\alpha)}\int_{0}^{t}\frac{f'(s)}{(t-s)^\alpha}\,\mathrm{d}s,\ 0<\alpha<1,\vspace{0.2cm}\\
  		
  		f'(t),\ \alpha=1.
  	\end{cases}
  \end{equation}
 On solving these equations, the explicit form of the state probabilities of the fractional PRF on $\mathbb{R}^2_+$ is derived. Some other fractional variants of the PRF, such as the space fractional and the space-time fractional PRFs are introduced and analyzed in \cite{Vishwakarma2025c}. More recently, Vishwakarma et al. \cite{Vishwakarma2025d} investigated some time-change general two parameter L\'evy processes with rectangular increments, driven by a bivariate process whose marginal components are independent inverse subordinators or increasing processes. As a particular case, they also studied time-changed variants of Poisson random fields with and without drifts. The spatial Skellam point process via PRF is introduced and studied \cite{Vishwakarma2025b}. In two parameter case, three different fractional variants of the same are investigated. Moreover, its the path integral over a rectangle in $\mathbb{R}^2_+$ is considered, and an explicit form of its Fourier transform is obtained.

Note that a Poisson point process indexed by the positive real line admits at most one arrival in a time interval of infinitesimal length. Also, from the aforementioned infinitesimal characterization of the PRF, it is clear that for any $A\in\mathcal{A}_d$ with infinitesimal Borel measure, the probability of observing more than one point in $A$ is negligible. An extension of the Poisson process is introduced in \cite{Crescenzo2016}, where within an infinitesimal time interval, multiple but finitely many arrivals can occur with positive probability. Different time-changed variants of the same process are studied by \cite{Kataria2022}. 

Motivated by all the above developments and their potential applications, in this paper, we introduce and study some spatial Poisson type point processes and their fractional variants. A brief  outline of the paper is as follows:
 
 In Section \ref{sec2}, we begin by defining a Poisson-type spatial point process on $\mathbb{R}^d_+$  for which it is possible to observe more than one point in a Borel subset of $\mathbb{R}^d_+$ with infinitesimal Lebesgue measure. That is, we introduce a non-negative integer valued random measure $\{M(A),\ A\in\mathcal{A}_d\}$ such that for disjoint rectangles $A_1,\dots,A_m\in\mathcal{A}_d$, $m\ge1$, $M(A_1),\dots,M(A_m)$ are independent random variables. Let $\mathbb{N}_0=\mathbb{N}\cup\{0\}$ denote the set of non-negative integers, and let $\Theta(k,n)=\{(n_1,n_2,\dots,n_k)\in\mathbb{N}^k_0:\sum_{j=1}^{k}jn_j=n\}$, $n\ge0$ be subsets of $\mathbb{N}^k_0$ for a fix $k\ge1$. Then, for $A\in\mathcal{A}_d$ with $|A|<\infty$, the distribution of $M(A)$ is given by
 \begin{equation*}
 	\mathrm{Pr}\{M(A)=n\}=\sum_{\Theta(k,n)}\prod_{j=1}^{k}\frac{(\lambda_j|A|)^{n_j}}{n_j!}e^{-\lambda_j|A|},\ n\ge0,
 \end{equation*}
 where $\lambda_1,\dots,\lambda_k$ are positive constants. We call it the generalized Poisson random field (GPRF). It is a spatial version of a generalized counting process introduced in \cite{Crescenzo2016}. We obtain its infinitesimal point probabilities. For $k=1$, the GPRF reduces to the PRF. Further, we establish that every GPRF is equal in distribution with a weighted sum of independent PRFs. Moreover, a compound Poisson random field representation of the GPRF is derived. 
 
 In Section \ref{gprf2sec}, we consider the GPRF on $\mathbb{R}^2_+$, and denote it by $\{M(s,t),\ (s,t)\in\mathbb{R}^2_+\}$. It is a two parameter L\'evy process with rectangular increments. We show that its probability generating function (pgf), $G(z,s,t)=\mathbb{E}z^{M(s,t)}$, $|z|\leq 1$ solves the following differential equation:
 \begin{equation*}
 	\frac{\partial^2}{\partial t\partial s}G(z,s,t)=\sum_{j=1}^{k}\lambda_j(z^j-1)G(z,s,t)+\sum_{j=1}^{k}\sum_{j'=1}^{k}\lambda_j\lambda_{j'}\sum_{r=1}^{j'}(z^{j+j'-r}-z^{j'-r})\frac{\partial}{\partial\lambda_1}G(z,s,t),
 \end{equation*}
 	with initial conditions $G(z,0,t)=G(z,s,0)=G(z,0,0)=1$ for all $s,t\ge0$. An approximation of the GPRF is obtained in the finite dimensional. Additionally, we discuss a thinning of the PRF on $\mathbb{R}^2_+$ into independent PRFs with reduced rates, which is then used for the thinning of a GPRF into independent GPRFs.
 	
 	In Section \ref{sec4}, we define a fractional variant of the GPRF on $\mathbb{R}^2_+$. For $0<\alpha,\beta<1$, let $\{L^\alpha(t),\ t\ge0\}$ and $\{L^\beta(t),\ge0\}$ be independent inverse stable subordinators that are independent of the GPRF $\{M(s,t),\ (s,t)\in\mathbb{R}^2_+\}$. We consider a time-changed two parameter process defined as follows:
\begin{equation*}
	M^{\alpha,\beta}(s,t)\coloneqq M(L^\alpha(s),L^\beta(t)),\ (s,t)\in\mathbb{R}^2_+.
\end{equation*}
We call $\{M^{\alpha,\beta}(s,t),\ (s,t)\in\mathbb{R}^2_+\}$, the fractional GPRF (FGPRF). We show that the point probabilities $P^{\alpha,\beta}(n,s,t)=\mathrm{Pr}\{M^{\alpha,\beta}(s,t)=n\}$, $n\ge0$ of FGPRF solve
\begin{align*}
	\frac{\partial^{\alpha+\beta}}{\partial t^\beta\partial s^\alpha}p^{\alpha,\beta}(n,s,t)=-\sum_{j=1}^{k}\lambda_j(I-B^j)\bigg(1+\sum_{j'=1}^{k}\lambda_{j'}\sum_{r=1}^{j'}B^{j'-r}\frac{\partial}{\partial\lambda_1}\bigg)p^{\alpha,\beta}(n,s,t),
\end{align*}
with $p^{\alpha,\beta}(0,0,0)=1$, where $B$ is the backward shift operator, defined as $B^jp^{\alpha,\beta}(n,s,t)=p^{\alpha,\beta}(n-j,s,t)$. We obtain its solutions which are given as follows:
\begin{equation*}
	p^{\alpha,\beta}(n,s,t)=\sum_{\Theta(k,n)}\prod_{j=1}^{k}\frac{(\lambda_js^\alpha t^\beta)^{n_j}}{n_j!}{}_2\Psi_2\Bigg[\begin{matrix}
		(\sum_{j=1}^{k}n_j+1,1),&(\sum_{j=1}^{k}n_j+1,1)\\\\
		(\alpha\sum_{j=1}^{k}n_j+1,\alpha),&(\beta\sum_{j=1}^{k}n_j+1,\beta)
	\end{matrix}\bigg|-\sum_{j=1}^{k}\lambda_js^\alpha t^\beta\Bigg], n\ge0,
\end{equation*}
where ${}_2\Psi_2$ is the generalized Wright function, defined as (see \cite{Kilbas2006})
\begin{equation}\label{gwright}
	{}_m\Psi_l\Bigg[\begin{matrix}
		(a_1,\alpha_1),&\dots,&(a_m,\alpha_m)\\\\
		(b_1,\beta_1),&\dots,&(b_l,\beta_l)
	\end{matrix}\bigg |x\Bigg]\coloneqq\sum_{r=0}^{\infty}\frac{\prod_{i=1}^{m}\Gamma(\alpha_i r+a_i)}{\prod_{j=1}^{l}\Gamma(\beta_j r+b_j)r!}x^r,\ x\in\mathbb{R},
\end{equation}
for $a_i, b_j\in\mathbb{R}$ and $\alpha_i, \beta_j\in\mathbb{R}-\{0\}$ for $1\leq i\leq m$, $1\leq j\leq l$.
Also, various other distributional properties of the FGPRF are derived.

In Section \ref{sec5}, we define a spatial Skellam type point process using the GPRF. Let $\mathcal{I}\subset\mathbb{R}-\{0\}$ be a finite subset, and let $\{M_i(A),\ A\in\mathcal{A}_d\}$, $i\in\mathcal{I}$ be independent GPRFs. We introduce an integer valued point process,
\begin{equation*}
	S(A)\coloneqq\sum_{i\in\mathcal{I}}iM_i(A),\ A\in\mathcal{A}_d.
\end{equation*}
We call $\{S(A),\ A\in\mathcal{A}_d\}$, the generalized Skellam Point process (GSPP). A compound Poisson field representation for the GSPP is obtained. For $\mathcal{I}=\{1,-1\}$, its point probabilities are derived explicitly. We also derive a finite dimensional weak approximation of the GSPP on $\mathbb{R}^2_+$. Later, a fractional variant of the GSPP is introduced and studied.
\section{Generalized Poisson random field}\label{sec2}
Here, we define a generalization of the Poisson random measure on finite dimensional Euclidean space. Let $\mathcal{A}_d=\{[\textbf{0},\textbf{x}],\ \textbf{x}\in\mathbb{R}^d_+\}$, $d\ge1$ be the collection of rectangles in $\mathbb{R}^d_+$, where $[\textbf{0},\textbf{x}]=\prod_{j=1}^{d}[0,x_j]$ for each $\textbf{x}=(x_1,\dots,x_d)\in\mathbb{R}^d_+$. 
\begin{definition} \label{gprfdef}
	Let $\{M(A),\ A\in\mathcal{A}_d\}$ be a non-negative integer valued random measure. We call it the generalized Poisson random field (GPRF) if\\
	\noindent (i) there exist positive constants $\lambda_1,\lambda_2,\dots,\lambda_k$ for a fix $k\ge1$ such that for any $A\in\mathcal{A}_d$ with $|A|<\infty$, the distribution of $M(A)$ is given by
	\begin{equation}\label{GPRFdist}
		\mathrm{Pr}\{M(A)=n\}=\sum_{\Theta(k,n)}\prod_{j=1}^{k}\frac{(\lambda_j|A|)^{n_j}}{n_j!}e^{-\lambda_j|A|},\ n\ge0,
	\end{equation}
	where $\Theta(k,n)=\{(n_1,n_2,\dots,n_k)\in\mathbb{N}^k_0:\sum_{j=1}^{k}jn_j=n\}$, and $|A|$ denotes the Lebesgue measure of $A$;\\
	\noindent (ii) For any collection $\{A_1,A_2,\dots,A_m\}\subset\mathcal{A}_d$ of disjoint sets,  $M(A_1),\dots, M(A_m)$, $m\ge2$ are independent of each other.
\end{definition} 
 We denote the GPRF by $\{M(A),\ A\in\mathcal{A}_d\}\sim GPRF\{\lambda_j\}_{1\leq j\leq k}$. Note that for $k=1$, GPRF reduces to the PRF as defined in Section \ref{sec1}. For $d=1$, it reduces to a generalized counting process on the positive real line, introduced and studied in \cite{Crescenzo2016}. In this case, it is equal in distribution to a generalized Skellam process in the sense of \cite{Cinque2025}. Note that (\ref{GPRFdist}) is a valid probability mass function, that is, it is non-negative and sums to one. The probability generating function $G(z,A)\coloneqq\mathbb{E}u^{M(A)}$, $|z|\leq1$ of GPRF is 
	\begin{equation}\label{GPRFpgf}
		G(z,A)=\sum_{n=0}^{\infty}z^n\mathrm{Pr}\{M(A)=n\}=\exp\bigg(\sum_{j=1}^{k}\lambda_j|A|(z^j-1)\bigg),\ A\in\mathcal{A}_d,\,|z|\leq1.
	\end{equation}
	Its mean and variance are given by $\mathbb{E}M(A)=\sum_{j=1}^{k}j\lambda_j|A|$ and $\mathbb{V}\mathrm{ar}M(A)=\sum_{j=1}^{k}j^2\lambda_j|A|$, respectively. For any $A,B\in\mathcal{A}_d$, its auto-covariance is $\mathbb{C}\mathrm{ov}(M(A),M(B))=\sum_{j=1}^{k}j^2\lambda_j|A\cap B|$.
\begin{remark}\label{rem21}
	In \cite{Baddeley2007}, the capacity functional of a point process $\{N(A),\ A\in\mathcal{A}_d\}$ is defined as $T_N(K)=1-\mathrm{Pr}\{N(K)=0\}$ for some compact set $K\subset\mathbb{R}^2_+$. It is noted that a point process can be characterized via its capacity functional. The capacity functional of GPRF is given by 
	$
		T_M(K)\coloneqq 1-e^{-(\lambda_1+\dots+\lambda_k)|K|}$.
\end{remark}
\begin{remark}
	 Let $\{M(A),\ A\in\mathcal{A}_d\}$  be a GPRF. Then, the variable $M(A)$ denotes the random number of points inside the rectangle $A\in\mathcal{A}_d$. Hence, similar to the PRF, its infinitesimal probabilities characterization is as follows: For any $A\in\mathcal{A}_d$ with infinitesimal Lebesgue measure, that is, $o(|A|)/|A|\to0$ as $|A|\to0$, the probability of exactly $j\in\{1,2,\dots,k\}$ many points in $A$ is given by $\mathrm{Pr}\{N(A)=j\}=\lambda_j|A|+o(|A|)$, the probability of no point in $A$ is $\mathrm{Pr}\{N(A)=0\}=1-\sum_{j=1}^{k}\lambda_j|A|+o(|A|)$, and the probability of more than $k$ points in $A$ is equal to $o(|A|)$. 
\end{remark}

Next result provides a unique representation for a particular type of GPRF.
\begin{theorem}\label{thmrep}
	Let $\lambda_1,\dots,\lambda_k$ be positive constants. Let $\{N_1(A),\ A\in\mathcal{A}_d\}, \dots, \{N_k(A),\ A\in\mathcal{A}_d\}$ be $k$-many mutually independent non-negative integer valued random measures. Then, the random measure $\{M(A),\ A\in\mathcal{A}_d\}=\{\sum_{j=1}^{k'}N_j(A),\ A\in\mathcal{A}_d\}\sim GPRF\{\lambda_j\}_{1\leq j\leq k'}$ for all $k'\leq k$ if and only if $\{N_1(A),\ A\in\mathcal{A}_d\},\dots, \{N_{k}(A),\ A\in\mathcal{A}_d\}$ are mutually independent PRFs with rate parameters $\lambda_1$, $\lambda_2$, $\dots$, $\lambda_k$, respectively.
\end{theorem}
\begin{proof}
	For every $j=1,2,\dots,k'\leq k$, let $\{N_j(A),\ A\in\mathcal{A}_d\}$ be a PRF with intensities $\lambda_j>0$ and set $M(A)=\sum_{j=1}^{k'}jN_j(A)$, $A\in\mathcal{A}_d$. Then, for any collection of disjoint sets $\{A_1,A_2,\dots,A_m\}\subset\mathcal{A}_d$, independence of $M(A_1)$, $M(A_2)$, $\dots$, $M(A_m)$ follows from the independence of $N_j(A)$'s. Further, the pgf $\mathbb{E}z^{\sum_{j=1}^{k'}jN_j(A)}=\exp(\sum_{j=1}^{k'}\lambda_j|A|(z^j-1))$ coincides with (\ref{GPRFpgf})	
	where we have used $\mathbb{E}z^{N_j(A)}=\exp(\lambda_j |A|(z-1))$, the pgf of PRF. Thus, the random measure $\{M(A),\ A\in\mathcal{A}_d\}$ is a GPRF.
	
	For the converse part, we use induction on the values of $k$. If $\{M(A),\ A\in\mathcal{A}_d\}$ is a GPRF then for $k=1$, it reduces to a PRF with some fixed positive intensity. Suppose the result holds for some $k=m\ge2$, that is, for a GPRF $\{M(A),\ A\in\mathcal{A}_d\}\sim GPRF\{\lambda_j\}_{1\leq j\leq m}$ there exist independent PRFs $\{N_j(A),\ A\in\mathcal{A}_d\}$, $j=1,2,\dots,m$ with intensity $\lambda_j$ such that $M(A)=\sum_{j=1}^{m}jN_j(A)$ for all $A\in\mathcal{A}_d$. Now, let $\{N_{m+1}(A),\ A\in\mathcal{A}_d\}$ be a point process independent of each PRF $\{N_j(A),\ A\in\mathcal{A}_d\}$ for $j=1,2,\dots,m$ such that $\tilde{M}(A)=\sum_{j=1}^{m}jN_j(A)+(m+1)N_{m+1}(A)$, $A\in\mathcal{A}_d$ where $\{\tilde{M}(A),\ A\in\mathcal{A}_d\}\sim GPRF\{\lambda_j\}_{1\leq j\leq m+1}$. Then, we need to show that $\{N_{m+1}(A),\ A\in\mathcal{A}_d\}$ is the PRF with intensity $\lambda_{m+1}>0$. By using (\ref{GPRFpgf}) and the induction hypothesis, the pgf of $\sum_{j=1}^{m+1}jN_j(A)$ is given by 
	\begin{equation*}
		\mathbb{E}z^{\sum_{j=1}^{m+1}jN_j(A)}=\exp\bigg(\sum_{j}^{m}\lambda_j|A|(z^j-1)\bigg)\mathbb{E}z^{(m+1)N_{m+1}(A)}.
	\end{equation*}
	On equating it with the pgf of $\tilde{M}(A)$, we get $\mathbb{E}z^{N_{m+1}(A)}=\exp(\lambda_{m+1}|A|(z-1))$, $|z|\leq1$. So, $N_{m+1}(A)$ has Poisson distribution with mean $\lambda_{m+1}|A|$. Moreover, for any disjoint rectangles $A_1$, $A_2,\dots,A_l\in\mathcal{A}_d$, by equating the pgf of $\sum_{i=1}^{l}\tilde{M}(A_i)$ to the pgf of $\sum_{i=1}^{l}\sum_{j=1}^{m+1}N_j(A_i)$, we get 
	\begin{equation*}
		\mathbb{E}z^{\sum_{i=1}^{l}N_{m+1}(A_i)}=\prod_{i=1}^{l}\exp(\lambda_{m+1}|A_i|(z-1))=\prod_{i=1}^{l}\mathbb{E}z^{N_{m+1}(A_i)},\ |z|\leq1.
	\end{equation*}
	Thus, the random variables $N_{m+1}(A_1)$, $N_{m+1}(A_2)$, $\dots$, $N_{m+1}(A_l)$ are independent of each other. This completes the proof.
\end{proof}
\begin{remark}\label{rem23}
	From Theorem \ref{thmrep}, we note that for any $\{M(A),\ A\in\mathcal{A}_d\}\sim GPRF\{\lambda_j\}_{1\leq j\leq k}$ there exist $k$-many independent Poisson random fields $\{N_1(A),\ A\in\mathcal{A}_d\}$, $\{N_2(A),\ A\in\mathcal{A}_d\}$, $\dots$, $\{N_k(A),\ A\in\mathcal{A}_d\}$ with transition intensities $\lambda_1$, $\lambda_2$, $\dots$, $\lambda_k$, respectively, such that 
$
		M(A)\overset{d}{=}\sum_{j=1}^{k}jN_j(A)$, $ A\in\mathcal{A}_d,
$
	where $\overset{d}{=}$ denotes the equality in distribution. Indeed, $\sum_{j=1}^{k}jN_j(A)$ is a generalized Skellam random field in the sense of \cite{Vishwakarma2025b}. Hence, from Proposition 3.1 of  \cite{Vishwakarma2025b}, its one dimensional distribution is equal to that of a compound Poisson random field, studied in \cite{Vishwakarma2025c}. In particular, if $\{X_r\}_{r\ge1}$ are iid random variables such that $\mathrm{Pr}\{X_1=j\}=\lambda_j/\sum_{j=1}^{k}\lambda_j$ for $j=1,\dots,k$, then the GPRF satisfies, $M(A)\overset{d}{=}\sum_{r=1}^{N(A)}X_r$, where $\{N(A),\ A\in\mathcal{A}_d\}$ is a PRF with rate parameter $\lambda_1+\dots+\lambda_k$, and it is independent of $\{X_r\}_{r\ge1}$.
\end{remark}
\subsection{Thinning of a GPRF} It is well known that a Poisson process on the positive real line can be thinned into independent Poisson processes with reduced rates by using a sequence of independent Bernoulli trials. A similar type of thinning for a generalized counting process with the possibilities of multiple arrivals (GPRF with $d=1$) is investigated in \cite{Dhillon2024}. Here, we study the thinning of a GPRF via independent Bernoulli trials. First, we discuss the thinning of a PRF. Let $\{Z_r\}_{r\ge1}$ be a sequence of independent and identically distributed (iid) Bernoulli random variables with success probability $p\in(0,1)$. Let $\{N(A),\ A\in\mathcal{A}_d\}$ be a PRF with rate parameter $\lambda>0$ as defined in Section \ref{sec1}, which is independent of $\{Z_r\}_{r\ge1}$. Let us now consider a random field defined as follows:
\begin{equation}\label{thndef1}
	S_N(A)\coloneqq\sum_{r=1}^{N(A)}Z_r,\ A\in\mathcal{A}_d.
\end{equation}
Note that $\{S_N(A),\ A\in\mathcal{A}_d\}$ is a compound Poisson random field as defined in \cite{Vishwakarma2025c}.
\begin{proposition}\label{prop21}
	Let $\{N(A),\ A\in\mathcal{A}_d\}$ be a PRF with rate parameter $\lambda>0$, and let $\{S_N(A),\ A\in\mathcal{A}_d\}$ be as defined in (\ref{thndef1}). Then, $\{S_N(A),\ A\in\mathcal{A}_d\}$ and $\{N(A)-S_N(A),\ A\in\mathcal{A}_d\}$ are PRFs with rate parameters $\lambda p$ and $\lambda(1-p)$, respectively. Moreover, $S_N(A)$ and $N(A)-S_N(A)$ are independent for all $A\in\mathcal{A}_d$.
\end{proposition}
\begin{proof}
	By a simple calculation, it can be shown that the characteristic function of $S_N(A)$ is given by
	$\mathbb{E}e^{iuS_N(A)}=\exp\big(\lambda p|A|(e^{iu}-1)\big)$, $u\in\mathbb{R}$. Hence, it is a Poisson variable with mean $\lambda p|A|$ for each $A\in\mathcal{A}_d$. Note that for any disjoint sets $A_1,\dots,A_n\in\mathcal{A}_d$, the random variables $Z_r$'s involved in $S_N(A_1),\dots,S_N(A_n)$ are indexed by the points of disjoint sets. Thus, these are independent using the independence of $Z_r$'s and $N(A_j)$'s. Therefore, $\{S_N(A),\ A\in\mathcal{A}_d\}$ is a PRF with rate parameter $\lambda p$. Similar arguments implies that $N(A)-S_N(A)$ is a PRF with rate parameter $\lambda(1-p)$. Moreover, for any $A\in\mathcal{A}_d$, the joint characteristic function of  $S_N(A)$ and $N(A)-S_N(A)$ is given by 
	\begin{align*}
		\mathbb{E}\exp\big(iuS_N(A)+iv(N(A)-S_N(A))\big)&=\mathbb{E}\bigg(\mathbb{E}\exp\big(iu\sum_{r=1}^{n}Z_r+iv(n-\sum_{r=1}^{n}Z_r)|N(A)=n\big)\bigg)\\
		&=\mathbb{E}e^{ivN(A)}\big(e^{i(u-v)}p+1-p\big)^{N(A)}\\
		&=\exp\big(\lambda|A|\big(e^{iu}p+(1-p)e^{iv}-1\big)\big)\\
		&=e^{\lambda|A|p(e^{iu}-1)}e^{\lambda|A|(1-p)(e^{iv}-1)},\ u,v\in\mathbb{R}.
	\end{align*} 
	This completes the proof.
\end{proof}

The following result extends the thinning of a PRF into more than two components in the obvious manner. 
\begin{proposition}
	Let $\{Z_r=(Z_r^{(1)},\dots, Z_r^{(k)})\}_{r\ge1}$, $k\ge1$ be a sequence of iid random vectors that is independent of the PRF $\{N(A),\ A\in\mathcal{A}_d\}$. Suppose $Z_r^{(1)}+\dots+Z_r^{(k)}=1$ for all $r\ge1$. Let us defined $S_N^{(j)}(A)\coloneqq\sum_{r=1}^{N(A)}Z_r^{(j)}$ for every $1\leq j\leq k$. Then, $\{S_N^{(j)}(A),\ A\in\mathcal{A}_d\}$'s are PRFs with some reduced rate parameters. Additionally, $S_N^{(1)}(A),\dots,S_N^{(k)}(A)$ are mutually independent for every $A\in\mathcal{A}_d$.
\end{proposition}

In view of Theorem \ref{thmrep}, the following result provides the thinning of a GPRF.
\begin{proposition}\label{prop23}
	Let $\{N_1(A),\ A\in\mathcal{A}_d\},\dots, \{N_k(A),\ A\in\mathcal{A}_d\}$ be independent PRFs with transition intensities $\lambda_1, \dots, \lambda_k$, respectively, and $\{M(A),\ A\in\mathcal{A}_d\}\sim GPRF\{\lambda_{j}\}_{1\leq j\leq k}$ be a GPRF, defined as $M(A)=\sum_{j=1}^{k}jN_j(A)$, $A\in\mathcal{A}_d$. Let $\{Z_r^{(j)}\}_{r\ge1}$, $j=1,\dots,k$ be independent sequences of iid random variables such that $\mathrm{Pr}\{Z_1^{(j)}=j\}=p_j$ and $\mathrm{Pr}\{Z_1^{(j)}=0\}=1-p_j$, where $p_j\in(0,1)$ for each $j=1,\dots,k$. Additionally, we assumed that $\{Z_r^{(j)}\}_{r\ge1}$'s are independent of $\{M(A),\ A\in\mathcal{A}_d\}$.  Let us define a random field,
$S_M(A)\coloneqq\sum_{j=1}^{k}\sum_{r=1}^{N_j(A)}Z_r^{(j)}$, $A\in\mathcal{A}_d$. Then, $\{S_M(A),\ A\in\mathcal{A}_d\}\sim GPRF\{\lambda_jp_j\}_{1\leq j\leq k}$ and $\{M(A)-S_M(A),\ A\in\mathcal{A}_d\}\sim\{\lambda_j(1-p_j)\}_{1\leq k\leq j}$. Moreover, $S_M(A)$ and $M(A)-S_M(A)$ are independent for each $A\in\mathcal{A}_d$. 
\end{proposition}
\begin{proof}
	For disjoint rectangles $A_1,\dots,A_m\in\mathcal{A}_d$, the independence of $S_M(A_1),\dots,S_M(A_m)$ follows from the independence of $\sum_{r=1}^{N_j(A_1)}Z_r^{(j)},\dots,\sum_{r=1}^{N_j(A_m)}Z_r^{(j)}$ for each $j=1,\dots,k$, and the independence of $N_j$'s and $Z_r^{(j)}$'s. Let $\{B_1^{(j)}\}_{r\ge1}$, $j=1,\dots,k$  be independent sequences of iid Bernoulli variables with success probability $p_j\in(0,1)$. Then, we note that $Z_1^{(j)}\overset{d}{=}jB_1^{(j)}$, and from Proposition \ref{prop21}, we have that $\sum_{r=1}^{N_j(A)}B_r^{(j)}$ is a PRF with rate parameter $\lambda_jp_j$ for each $j=1,\dots,k$. So, from Remark \ref{rem23}, the distribution of $S_M(A)$ coincides to that of GPRF.  Thus, $\{S_M(A),\ A\in\mathcal{A}_d\}\sim GPRF\{\lambda_jp_j\}_{1\leq j\leq k}$. Similar arguments gives that $\{M(A)-S_M(A),\ A\in\mathcal{A}_d\}\sim GPRF\{\lambda_j(1-p_j)\}_{1\leq j\leq k}$. For the independence of $S_M(A)$ and $M(A)-S_M(A)$, it is enough to show the independence of $\sum_{r=1}^{N_j(A)}Z_r^{(j)}$ and $jN_j(A)-\sum_{r=1}^{N_j(A)}Z_r^{(j)}$ for each $j=1,\dots,k$, which follows along similar lines to that of Proposition \ref{prop21}. This completes the proof.
\end{proof}

\section{GPRF on $\mathbb{R}^2_+$}\label{gprf2sec}
 We now consider the GPRF on $\mathbb{R}^2_+$ as a two parameter L\'evy process. Let us now recall the definition of a two parameter L\'evy process with rectangular increments. The characterization of such processes is given in \cite{Straf1972}.
 
 A random process $\{X(s,t),\ (s,t)\in\mathbb{R}^2_+\}$, defined as $X(s,t)\coloneqq X((0,s]\times(0,t])$ is called the two parameter L\'evy process if\\
	\noindent (i) $X(0,t)=X(s,0)=0$ almost surely;\\
	\noindent (ii) for every $(s,t)\preceq(s',t')$, the rectangular increment of $X(s,t)$ over a rectangle $(s,s']\times(t,t']\subset\mathbb{R}^2_+$ is defined as
	\begin{equation}\label{rinc}
		\Delta_{s,t}X(s',t')=X((s,s']\times(t,t'])\coloneqq X(s',t')-X(s',t)-X(s,t')+X(s,t).
	\end{equation}
	Moreover, the distribution of increment $\Delta_{s,t}X(s',t')$ is equal to that of $X(s'-s,t'-t)$;\\
	\noindent (iii) for any disjoint rectangles $A_1,\dots,A_m$ in $\mathbb{R}^2_+$ of type $(s,s']\times(t,t']$, the random variables $X(A_1),\dots,X(A_m)$ are mutually independent.

For $d=2$, the GPRF in Definition \ref{gprfdef} reduces to a random fields on $\mathbb{R}^2_+$ whose increments can be defined as (\ref{rinc}). That is, a two parameter process $\{M(s,t),\ (s,t)\in\mathbb{R}^2_+\}$ will be called the two parameter GPRF with rate parameters $\lambda_1,\dots,\lambda_k>0$ if $M(0,t)=M(s,0)=0$ almost surely. For any $(0,0)\preceq(s,t)\prec(s',t')$, the distribution of $\Delta_{s,t}M(s',t')$ is given by 
\begin{equation*}
	\mathrm{Pr}\{\Delta_{s,t}M(s',t')=n\}=\sum_{\Theta(k,n)}\prod_{j=1}^{k}\frac{(\lambda_j(s'-s)(t'-t))^{n_j}}{n_j!}e^{-\lambda_j(s'-s)(t'-t)},\ n\ge0,
\end{equation*}
where $\Theta(k,n)$ is as in Definition \ref{gprfdef}. Thus, the GPRF on $\mathbb{R}^2_+$ is a two parameter L\'evy process with rectangular increment. We denote it by $\{M(s,t),\ (s,t)\in\mathbb{R}^2_+\}\sim GPRF\{\lambda_j\}_{1\leq j\leq k}$.
\begin{remark}\label{rem41}
	From Remark \ref{rem23}, it follows that for every GPRF, $\{M(s,t), (s,t)\in\mathbb{R}^2_+\}\sim GPRF\{\lambda_j\}_{1\leq j\leq k}$, there exist independent PRFs $\{N_j(s,t),\ (s,t)\in\mathbb{R}^2_+\}$, $j=1,\dots,k$ with rate parameter $\lambda_j$'s such that $M(s,t)\overset{d}{=}\sum_{j=1}^{k}jN_j(s,t)$, $(s,t)\in\mathbb{R}^2_+$. Thus, the mean and variance of $M(s,t)$ are given by $\mathbb{E}M(s,t)=\sum_{j=1}^{k}j\lambda_jst$ and $\mathbb{V}\mathrm{ar}M(s,t)=\sum_{j=1}^{k}j^2\lambda_jst$, respectively. For any $(s,t)$ and $(s',t')$ in $\mathbb{R}^2_+$, its auto covariance function is  
	\begin{equation*}
		\mathbb{C}\mathrm{ov}(M(s,t),M(s',t'))=\sum_{j=1}^{k}j^2\lambda_j\min\{s,s'\}\min\{t,t'\}.
	\end{equation*} 
\end{remark}
\begin{remark}
	The one dimensional distribution of GPRF $\{M(s,t),\ (s,t)\in\mathbb{R}^2_+\}$ is given by 
	\begin{equation*}
		p(n,s,t)\coloneqq\mathrm{Pr}\{M(s,t)=n\}=\sum_{\Theta(k,n)}\prod_{j=1}^{k}\frac{(\lambda_jst)^{n_j}}{n_j!}e^{-\lambda_jst},\ n\ge0,
	\end{equation*}
	which solves 
	\begin{equation*}
		\frac{\partial}{\partial s}p(n,s,t)=-\sum_{j=1}^{k}\lambda_jtp(n,s,t)+\sum_{j=1}^{k}\lambda_jtp(n-j,s,t),\ n\ge0,
	\end{equation*}
	with $p(0,0,t)=1$ for all $t\ge0$. 
\end{remark}

	From (\ref{GPRFpgf}), the pgf of $M(s,t)$ is 
	\begin{equation}\label{gprf2pgf}
		G(z,s,t)\coloneqq\mathbb{E}z^{M(s,t)}=\exp\bigg(\sum_{j=1}^{k}\lambda_jst(z^j-1)\bigg),\ |z|\leq 1.
	\end{equation}
	On taking the derivative of (\ref{gprf2pgf}) with respect to $s$, we get
	\begin{equation*}
		\frac{\partial}{\partial s}G(z,s,t)=\sum_{j=1}^{k}\lambda_j(z^j-1)t\exp\bigg(\sum_{j=1}^{k}\lambda_jst(z^j-1)\bigg),
	\end{equation*}
	whose derivative with respect to $t$ yields 
	\begin{align}
		\frac{\partial^2}{\partial t\partial s}G(z,s,t)&=\sum_{j=1}^{k}\lambda_j(z^j-1)G(z,s,t)+\bigg(\sum_{j=1}^{k}\lambda_j(z^j-1)\bigg)^2stG(z,s,t)\nonumber\\
		&=\sum_{j=1}^{k}\lambda_j(z^j-1)G(z,s,t)+\sum_{j=1}^{k}\lambda_j(z^j-1)\bigg(\sum_{j=1}^{k}\lambda_j\sum_{r=1}^{j}z^{j-r}\bigg)(z-1)stG(z,s,t)\nonumber\\
		&=\sum_{j=1}^{k}\lambda_j(z^j-1)G(z,s,t)+\sum_{j=1}^{k}\lambda_j(z^j-1)\bigg(\sum_{j'=1}^{k}\lambda_{j'}\sum_{r=1}^{j'}z^{j'-r}\bigg)\frac{\partial}{\partial\lambda_1}G(z,s,t)\nonumber\\
		&=\sum_{j=1}^{k}\lambda_j(z^j-1)G(z,s,t)+\sum_{j=1}^{k}\sum_{j'=1}^{k}\lambda_j\lambda_{j'}\sum_{r=1}^{j'}(z^{j+j'-r}-z^{j'-r})\frac{\partial}{\partial\lambda_1}G(z,s,t),\label{gprfpgfeq}
	\end{align}
	with initial conditions $G(z,0,t)=G(z,s,0)=G(z,0,0)=1$ for all $s,t\ge0$.

The following result gives a finite dimensional convergence to the two parameter GPRF.
\begin{proposition}
	For $m\ge1$, $l\ge1$ and $n\ge1$, let $p_{m,l,j}^{(n)}\in(0,1)$ for each $j=1,\dots,k$ such that $\sum_{j=1}^{k}p_{m,l,j}^{(n)}<1$. Let $\{X_{m,l}^{(n)}\}_{m,l,n\ge1}$ be a triple indexed sequence of random variables independent across $m$, $l$ and $n$ such that
	\begin{equation*}
		X_{m,l}^{(n)}\coloneqq\begin{cases}
			j\in\{1,\dots,k\},\ \text{with probability}\ p_{m,l,j}^{(n)},\\
			0,\ \text{with probability}\ 1-\sum_{j=1}^{k}p_{m,l,j}^{(n)}.
		\end{cases}
	\end{equation*}
	Let us consider a two parameter process $\{Y^{(n)}(s,t),\ (s,t)\in\mathbb{R}^2_+\}$, $n\ge1$ defined as follows:
	\begin{equation*}
		Y^{(n)}(s,t)\coloneqq\sum_{m=1}^{[ns]}\sum_{l=1}^{[nt]}X_{m,l}^{(n)},\ (s,t)\in\mathbb{R}^2_+.
	\end{equation*}
	If $\max_{1\leq m,l\leq n}p_{m,l,j}^{(n)}\rightarrow0$ and
	\begin{equation*}
		\sum_{m=1}^{[ns]}\sum_{l=1}^{[nt]}p_{m,l,j}^{(n)}\rightarrow \lambda_jst\ \text{for all}\ (s,t)\in\mathbb{R}^2_+\ \text{and}\ j=1,\dots,k,
	\end{equation*}
	for some positive constant $\lambda_1,\dots,\lambda_k$, as $n\rightarrow\infty$, then for any $(0,0)\prec(s_1,t_1)\preceq\dots\preceq(s_r,t_r)$, $r\ge1$, we have
	\begin{equation*}
		(Y^{(n)}(s_1,t_1),\dots, Y^{(n)}(s_r,t_r))\overset{d}{\rightarrow}(M(s_1,t_1),\dots,M(s_1,t_r))\ \text{as}\ n\rightarrow\infty,
	\end{equation*}
	where $\{M(s,t),\ (s,t)\in\mathbb{R}^2_+\}\sim GPRF\{\lambda_j\}_{1\leq j\leq k}$, and $\overset{d}{\rightarrow}$ denotes the convergence in distribution.
\end{proposition}
\begin{proof}
	For $\{M(s,t), (s,t)\in\mathbb{R}^2_+\}\sim GPRF\{\lambda_j\}_{1\leq j\leq k}$, there exist $k$-many independent PRFs $\{N_j(s,t),\ (s,t)\in\mathbb{R}^2_+\}$, $j=1,\dots,k$ with rate parameter $\lambda_j$'s such that $M(s,t)\overset{d}{=}\sum_{j=1}^{k}jN_j(s,t)$, $(s,t)\in\mathbb{R}^2_+$. Moreover, $\sum_{j=1}^{k}jN_j(s,t)$ is a generalized Skellam random field on $\mathbb{R}^2_+$ in the sense of \cite{Vishwakarma2025b}. Therefore, the result follows from Theorem 3.1 of \cite{Vishwakarma2025b}. 
\end{proof}
\subsection{Thinning of GPRF on $\mathbb{R}^2_+$} Similar to PRF on $\mathbb{R}^d_+$, we now investigate the thinning of a two parameter PRF into independent PRFs with reduced rate parameters.
\begin{theorem}\label{thm22}
	Let $\{N(s,t),\ (s,t)\in\mathbb{R}^2_+\}$ be a PRF with rate parameter $\lambda>0$. Let $\{Z_r\}_{r\ge1}$ be a sequence of iid Bernoulli random variables with parameter $p\in(0,1)$ which is independent of $\{N(s,t),\ (s,t)\in\mathbb{R}^2_+\}$. Let us consider a two parameter process,
	$S_N(s,t)\coloneqq\sum_{r=1}^{N(s,t)}Z_r$, $(s,t)\in\mathbb{R}^2_+$.
	Then, $\{S_N(s,t),\ (s,t)\in\mathbb{R}^2_+\}$ and $\{N(s,t)-S_N(s,t),\ (s,t)\in\mathbb{R}^2_+\}$ are independent PRFs with rate parameters $\lambda p$ and $\lambda(1-p)$, respectively. 
\end{theorem} 
\begin{proof}
	Note that $\{S_N(s,t),\ (s,t)\in\mathbb{R}^2_+\}$ and $\{N(s,t)-S_N(s,t),\ (s,t)\in\mathbb{R}^2_+\}$ are two parameter compound Poisson random fields as defined in \cite{Vishwakarma2025c}. From Proposition 3.1 of \cite{Vishwakarma2025c}, these are two parameter L\'evy processes. Thus, from Proposition \ref{prop21}, it follows that $S_N(s,t)$ and $N(s,t)-S_N(s,t)$ are PRFs with rate parameters $\lambda pst$ and $\lambda(1-p)st$, respectively.	
	  To prove the independence of $\{S_N(s,t),\ (s,t)\in\mathbb{R}^2_+\}$ and $\{N(s,t)-S_N(s,t),\ (s,t)\in\mathbb{R}^2_+\}$, we need to show that for any $m,n\ge1$, $(s_1,t_1),\dots,(s_m,t_m)$ and $(s_1',t_1'),\dots,(s_n',t_n')$ in $\mathbb{R}^2_+$, $\{S_N(s_1,t_1),\dots,S_N(s_m,t_m)\}$ and $\{N(s_1',t_1')-S_N(s_1',t_1'),\dots,N(s_n',t_n')-S_N(s_n',t_n')\}$ are independent collections. To show that it is sufficient to prove the independence of $S_N(s,t)$ and $N(s',t')-S_N(s',t')$ for all $(s,t)$ and $(s',t')$ in $\mathbb{R}^2_+$. Note that for any $(s,t)$ and $(s',t')$, we have the following four possible cases:\\
	\noindent (i) $s\leq s'$, $t\leq t'$,\\
	\noindent (ii) $s\leq s'$, $t\ge t'$,\\
	\noindent (iii) $s\ge s'$, $t\leq t'$,\\
	\noindent (iv) $s\ge s'$, $t\ge t'$.
	
	Set $\tilde{S}_N(s,t)=N(s,t)-S_N(s,t)=\sum_{r=1}^{N(s,t)}(1-Z_r)$ for all $(s,t)\in\mathbb{R}^2_+$. For disjoint rectangles $(s_1,s_1']\times(t_1,t_1']$ and $(s_2,s_2']\times(t_2,t_2']$, the random variables $S_N((s_1,s_1']\times(t_1,t_1'])$ and $\tilde{S}_N((s_2,s_2']\times(t_2,t_2'])$ are independent since they only involve $Z_r$'s indexed by distinct points of disjoint rectangles. Also, from Proposition \ref{prop21}, $S_N((s,s']\times(t,t'])$ and $\tilde{S}_N((s,s']\times(t,t'])$ are independent for all $(s,s']\times(t,t']\subset\mathbb{R}^2_+$. If $s\leq s'$, $t\leq t'$, then $\tilde{S}_N(s',t')=\Delta_{s,t}\tilde{S}_N(s',t')+\Delta_{s,0}\tilde{S}_N(s',t)+\Delta_{0,t}\tilde{S}_N(s,t')+\tilde{S}_N(s,t)$, which is a sum of increments independent of $S_N(s,t)$.
	Now, suppose $s\leq s'$ and $t\ge t'$. Then, $S_N(s,t)=\Delta_{0,t'}S_N(s,t)+S_N(s,t')$ and $\tilde{S}_N(s',t')=\Delta_{s,0}\tilde{S}_N(s',t')+\tilde{S}_N(s,t')$, which are independent of each other. The independence for other two cases can be shown along similar lines. This completes the proof.
\end{proof}

The proofs of the following results are similar to that of Theorem \ref{thm22} and Proposition \ref{prop23}. 
\begin{theorem}
For a fix $k\ge1$, let $\{Z_r=(Z_r^{(1)},\dots, Z_r^{(k)})\}_{r\ge1}$ be a sequence of iid $k$-dimensional random vectors with finite mean, that is independent of  a PRF $\{N(s,t),\ (s,t)\in\mathbb{R}^2_+\}$. Suppose $Z_r^{(1)}+\dots+Z_r^{(k)}=1$ for all $r\ge1$, and $S_N^{(j)}(s,t)\coloneqq\sum_{r=1}^{N(s,t)}Z_r^{(j)}$ for every $1\leq j\leq k$. Then, $\{S_N^{(j)}(s,t),\ (s,t)\in\mathbb{R}^2_+\}$, $j=1,\dots,k$ are independent PRFs.
\end{theorem}

\begin{theorem}
	Let $\{N_1(s,t),\ (s,t)\in\mathbb{R}^2_+\},\dots,\{N_k(s,t),\ (s,t)\in\mathbb{R}^2_+\}$ be independent PRFs with transition rates $\lambda_1$, $\lambda_2$, $\dots$, $\lambda_k$, respectively, and let $\{M(s,t)=\sum_{j=1}^{k}jN_j(s,t),\ (s,t)\in\mathbb{R}^2_+\}\sim GPRF\{\lambda_{j}\}_{1\leq j\leq k}$. Also, let $\{Z_r^{(j)}\}_{r\ge1}$, $j=1,\dots,k$ be independent sequences of iid random variables as defined in Proposition \ref{prop23}, that are independent of $\{M(s,t),\ (s,t)\in\mathbb{R}^2_+\}$.  Suppose
	$S_M(s,t)\coloneqq\sum_{j=1}^{k}\sum_{r=1}^{N_j(s,t)}Z_r^{(j)}$, $(s,t)\in\mathbb{R}$. Then, $\{S_M(s,t),\ (s,t)\in\mathbb{R}^2_+\}\sim GPRF\{\lambda_jp_j\}_{1\leq j\leq k}$ and $\{M(s,t)-S_M(s,t),\ (s,t)\in\mathbb{R}^2_+\}\sim GPRF\{\lambda_j(1-p_j)\}_{1\leq k\leq j}$ are independent GPRFs.
\end{theorem}
\subsection{Integrals of GPRF on $\mathbb{R}^2_+$} Let $\{M(s,t),\ (s,t)\in\mathbb{R}^2_+\}\sim GPRF\{\lambda_j\}_{1\leq j\leq k}$. We define the following integral of GPRF:
\begin{equation}\label{gprfint}
	\int_{0}^{t}\int_{0}^{s}M(x,y)\,\mathrm{d}x\,\mathrm{d}y,\ (s,t)\in\mathbb{R}^2_+.
\end{equation}
Its mean is equal to
$\sum_{j=1}^{k}j\lambda_j(st)^2/4$. The GPRF is a two parameter L\'evy process and its characteristic function is given by
$\mathbb{E}e^{\iota uM(s,t)}=\exp\bigg(st\sum_{j=1}^{k}\lambda_j(e^{\iota ju}-1)\bigg)$, $u\in\mathbb{R}$.
From Remark 5.1 and Eq. (5.4) of \cite{Vishwakarma2025b}, the characteristic function of (\ref{gprfint}) has the following representation:
\begin{equation*}
	\mathbb{E}\exp\bigg(\iota u\int_{0}^{t}\int_{0}^{s}M(x,y)\,\mathrm{d}x\,\mathrm{d}y\bigg)=\exp\bigg(st\sum_{j=1}^{k}\lambda_j\int_{0}^{1}\int_{0}^{1}(e^{\iota jstuxy}-1)\,\mathrm{d}x\,\mathrm{d}y\bigg),\ u\in\mathbb{R}.
\end{equation*}
Using it, the variance of (\ref{gprfint}) equals
$\sum_{j=1}^{k}j^2\lambda_j(st)^3/9$, $(s,t)\in\mathbb{R}^2_+$. 

Let $\{X_r\}_{r\ge1}$ be the sequence as defined in Remark \ref{rem23}, which is independent of a two parameter PRF $\{N(s,t),\ (s,t)\in\mathbb{R}^2_+\}$ with rate parameter $\lambda_1+\dots+\lambda_k$. Then, from Remark 5.2 of \cite{Vishwakarma2025b}, the following equality holds:
	\begin{equation*}
		\int_{0}^{t}\int_{0}^{s}M(x,y)\,\mathrm{d}x\,\mathrm{d}y\overset{d}{=}st\sum_{r=1}^{N(s,t)}X_rU_r,\ (s,t)\in\mathbb{R}^2_+,
	\end{equation*}
	where $U_r$'s are independent uniform variables over $[0,1]\times[0,1]$, that are independent of both $N(s,t)$ and $X_r$'s. 

\section{Fractional GPRF on $\mathbb{R}^2_+$} \label{sec4}
Different fractional variants of the GPRF indexed by positive real line are introduced and studied in \cite{Kataria2022}. Recently, a fractional variant of the PRF on $\mathbb{R}^2_+$ is introduced and studied in \cite{Kataria2024}, where it is defined as a process whose one dimensional distribution solves a system of fractional differential equations. Its time-changed and different other characterizations can be found in \cite{Aletti2018, Leonenko2015}. We now define a fractional variant of the GPRF as a time-changed two parameter L\'evy process. First, we recall the definition of an inverse stable subordinator.

A non-negative real-valued L\'evy process with almost surely non-decreasing sample path is called subordinator. A subordinator $\{H^\alpha(t),\ t\ge0\}$ is called the stable subordinator with index $\alpha\in(0,1)$ if its Laplace transform is given by $\mathbb{E}e^{-uH^\alpha(t)}=e^{-tu^\alpha}$, $u>0$. A hitting time process $\{L^\alpha(t),\ t\ge0\}$ defined as 
\begin{equation}\label{insubdef}
	L^\alpha(t)\coloneqq\inf\{s\ge0: H^\alpha(s)\ge t\}
\end{equation}
is called an inverse stable subordinator of index $\alpha$. The Laplace transform of its density with respect to the time variable is given by (see \cite{Meerscheart2013})
\begin{equation}\label{insublap}
	\int_{0}^{\infty}e^{-ut}\mathrm{Pr}\{L^\alpha(t)\in\mathrm{d}x\}=u^{\alpha-1}e^{-u^\alpha x}\,\mathrm{d}x,\ u>0,\ x\ge0,
\end{equation}
where $\mathrm{d}x$ denotes the infinitesimal length in $\mathbb{R}$. Further, we assume that $L^\alpha(t)|_{\alpha=1}=t$ almost surely.

Let $\{L^\alpha(t),\ t\ge0\}$ and $\{L^\beta(t),\ t\ge0\}$ be independent inverse stable subordinators of indices $\alpha\in(0,1)$ and $\beta\in(0,1)$, respectively. Suppose $\{M(s,t),\ (s,t)\in\mathbb{R}^2_+\}\sim GPRF\{\lambda_j\}_{1\leq j\leq k}$. We consider a time-changed two parameter process $\{M^{\alpha,\beta}(s,t),\ (s,t)\in\mathbb{R}^2_+\}$ defined as follows:
\begin{equation*}
	M^{\alpha,\beta}(s,t)\coloneqq M(L^\alpha(s),L^\beta(t)),\ 0<\alpha,\beta\leq1,
\end{equation*}
where are the component processes are assumed to be independent of each other.
We called it the fractional GPRF (FGPRF) and denote it by $\{M^{\alpha,\beta}(s,t),\ (s,t)\in\mathbb{R}^2_+\}\sim FGPRF\{\lambda_j\}_{1\leq j\leq k}$. For $\alpha=1$ and $\beta=1$, it reduces to the two parameter GPRF. For $k=1$, the FGPRF reduces to the fractional Poisson random field introduced in \cite{Leonenko2015}.
\begin{proposition}
	The pgf of FGPRF is given by
	\begin{equation}\label{fgprfpgf}
		G^{\alpha,\beta}(z,s,t)\coloneqq\mathbb{E}z^{M^{\alpha,\beta}(s,t)}={}_2\Psi_2\Bigg[\begin{matrix}
			(1,1),&(1,1)\\\\
			(1,\alpha),&(\beta,1)
		\end{matrix}\bigg|\sum_{j=1}^{k}\lambda_j(z^j-1)s^\alpha t^\beta\Bigg],\ |z|\leq 1,
	\end{equation}
	where ${}_2\Psi_2$ is the generalized Wright function as defined in (\ref{gwright}).
	
	Moreover, the pgf (\ref{fgprfpgf}) solves the following fractional partial differential equation:
	\begin{equation}\label{fgprfpgfeq}
		\frac{\partial^{\alpha+\beta}}{\partial t^\beta\partial s^\alpha}G^{\alpha,\beta}(z,s,t)=\sum_{j=1}^{k}\lambda_j(z^j-1)G^{\alpha,\beta}(z,s,t)+\sum_{j=1}^{k}\sum_{j'=1}^{k}\lambda_j\lambda_{j'}\sum_{r=1}^{j'}(z^{j+j'-r}-z^{j'-r})\frac{\partial}{\partial\lambda_1}G^{\alpha,\beta}(z,s,t),
	\end{equation}
	with initial conditions $G^{\alpha,\beta}(z,0,t)=G^{\alpha,\beta}(z,s,0)=G^{\alpha,\beta}(z,0,0)=1$ for all $s,t\ge0$, where the partial derivatives are in the sense of Caputo fractional derivative as defined in (\ref{caputoder}). 
\end{proposition}
\begin{proof}
	On using (\ref{gprf2pgf}), we have
	\begin{align}
		G^{\alpha,\beta}(z,s,t)&=\iint_{\mathbb{R}^2_+}G(z,x,y)\,\mathrm{Pr}\{L^\alpha(s)\in\mathrm{d}x\}\mathrm{Pr}\{L^\beta(t)\in\mathrm{d}y\}\label{pgfeqpf1}\\
		&=\iint_{\mathbb{R}^2_+}\exp\bigg(\sum_{j=1}^{k}\lambda_jxy(z^j-1)\bigg)\,\mathrm{Pr}\{L^\alpha(s)\in\mathrm{d}x\}\mathrm{Pr}\{L^\beta(t)\in\mathrm{d}y\}.\nonumber
	\end{align}
	Its double Laplace transform  is given by
	\begin{align}
		\iint_{\mathbb{R}^2_+}e^{-us-vt}G^{\alpha,\beta}(z,s,t)\,\mathrm{d}s\,\mathrm{d}t&=u^{\alpha-1}v^{\beta-1}\iint_{\mathbb{R}^2_+}\exp\bigg(\sum_{j=1}^{k}\lambda_jxy(z^j-1)\bigg)e^{-u^\alpha x}e^{-v^\beta y}\,\mathrm{d}x\,\mathrm{d}y\nonumber\\
		&=\int_{0}^{\infty}v^{\beta-1}\frac{u^{\alpha-1}}{u^\alpha-\sum_{j=1}^{k}\lambda_jy(z^j-1)}e^{-v^\beta y}\,\mathrm{d}y,\label{prop41pf1}
	\end{align} 
	where we have used (\ref{insublap}) to get the first equality. On taking the inverse Laplace transform of (\ref{prop41pf1}) with respect to variable $u$, and using the following result (see \cite{Kilbas2006}):
	\begin{equation}\label{opmllap}
		\int_{0}^{\infty}e^{-ux}E_{\alpha,1}(cx)\,\mathrm{d}x=\frac{u^{\alpha-1}}{u^\alpha-c},\ c\in\mathbb{R},\ u>0,
	\end{equation}
	we get
	\begin{align}
		\int_{0}^{\infty}e^{-vt}G^{\alpha,\beta}(z,s,t)\,\mathrm{d}t&=\int_{0}^{\infty}v^{\beta-1}E_{\alpha,1}\bigg(\sum_{j=1}^{k}\lambda_jy(z^j-1)s^\alpha\bigg)e^{-v^\beta y}\,\mathrm{d}y\nonumber\\
		&=\sum_{r=0}^{\infty}\frac{(\sum_{j=1}^{k}\lambda_j(z^j-1)s^\alpha)^r}{\Gamma(\alpha r+1)}v^{\beta-1}\int_{0}^{\infty}y^{r}e^{-v^\beta y}\,\mathrm{d}y\nonumber\\
		&=\sum_{r=0}^{\infty}\frac{(\sum_{j=1}^{k}\lambda_j(z^j-1)s^\alpha)^r\Gamma(r+1)}{\Gamma(\alpha r+1)v^{\beta r+1}},\label{fgprfpgfpf1}
	\end{align}
	where $E_{\alpha,1}(\cdot)$ is the one parameter Mittag-Leffler function defined as (see \cite{Kilbas2006})
	\begin{equation}\label{opml}
		E_{\alpha,1}(x)\coloneqq\sum_{r=0}^{\infty}\frac{x^r}{\Gamma(\alpha r+1)},\ \alpha>0,\ x\in\mathbb{R}.
	\end{equation} 
	Note that the interchange of integral and sum in the second step follows from Fubini's theorem.
	On taking the inverse Laplace transform of (\ref{fgprfpgfeqpf1}), we get
	\begin{equation*}
		G^{\alpha,\beta}(z,s,t)=\sum_{r=0}^{\infty}\frac{(\sum_{j=1}^{k}\lambda_j(z^j-1)s^\alpha t^\beta)^r\Gamma(r+1)}{\Gamma(\alpha r+1)\Gamma(\beta r+1)}.
	\end{equation*}
	This gives the required pgf (\ref{fgprfpgf}).
	
	Now, we derive the governing equation of (\ref{fgprfpgf}). On taking the Laplace transform with respect to $s$ on both sides of  (\ref{fgprfpgfeq}) and using the following result (see \cite{Kilbas2006}):	 
	\begin{equation}\label{caputolap}
		\int_{0}^{\infty}e^{-ut}\bigg(\frac{\mathrm{d}^\alpha}{\mathrm{d}t^\alpha}f(t)\bigg)\,\mathrm{d}t=u^\alpha\int_{0}^{\infty}e^{-ut}f(t)\,\mathrm{d}t-u^{\alpha-1}f(0^+),\ u>0,
	\end{equation}
	we get
	\begin{align*}
		u^\alpha\int_{0}^{\infty}e^{-us}\frac{\partial^\beta}{\partial t^\beta}G^{\alpha,\beta}(z,s,t)\,\mathrm{d}s&=\sum_{j=1}^{k}\lambda_j(z^j-1)\int_{0}^{\infty}e^{-us}G^{\alpha,\beta}(z,s,t)\,\mathrm{d}s\\
		&\ \ +\sum_{j=1}^{k}\sum_{j'=1}^{k}\lambda_j\lambda_{j'}\sum_{r=1}^{j'}(z^{j+j'-r}-z^{j'-r})\int_{0}^{\infty}e^{-us}\frac{\partial}{\partial\lambda_1}G^{\alpha,\beta}(z,s,t)\,\mathrm{d}s,
	\end{align*}
	where we have used $\frac{\partial^\beta}{\partial t^\beta}G^{\alpha,\beta}(z,0,t)=0$. Its Laplace transform with respect to $t$ yields
	\begin{align}
		u^\alpha v^{\beta}\iint_{\mathbb{R}^2_+}e^{-us-vt}&G^{\alpha,\beta}(z,s,t)\,\mathrm{d}s\,\mathrm{d}t-u^{\alpha-1}v^{\beta-1}\nonumber\\
		&=\sum_{j=1}^{k}\lambda_j(z^j-1)\iint_{\mathbb{R}^2_+}e^{-us-vt}G^{\alpha,\beta}(z,s,t)\,\mathrm{d}s\,\mathrm{d}t\nonumber\\
		&\ \ +\sum_{j=1}^{k}\sum_{j'=1}^{k}\lambda_j\lambda_{j'}\sum_{r=1}^{j'}(z^{j+j'-r}-z^{j'-r})\iint_{\mathbb{R}^2_+}e^{-us-vt}\frac{\partial}{\partial\lambda_1}G^{\alpha,\beta}(z,s,t)\,\mathrm{d}s\,\mathrm{d}t,\label{fgprfpgfeqpf1}
	\end{align}
	where we have used $G^{\alpha,\beta}(z,s,0)=1$. Now, we show that the double Laplace transform \begin{equation}\label{fgprfpgfpf2}
		\iint_{\mathbb{R}^2_+}e^{-us-vt}G^{\alpha,\beta}(z,s,t)\,\mathrm{d}s\,\mathrm{d}t=u^{\alpha-1}v^{\beta-1}\iint_{\mathbb{R}^2_+}G(z,x,y)e^{-u^\alpha x-v^\beta y}\,\mathrm{d}x\,\mathrm{d}y
	\end{equation}
	of (\ref{pgfeqpf1})	satisfies (\ref{fgprfpgfeqpf1}), where we have used (\ref{insublap}). On substituting (\ref{fgprfpgfpf2}) in the right hand side of (\ref{fgprfpgfeqpf1}) and using (\ref{gprfpgfeq}), we get
	\begin{align*}
		\sum_{j=1}^{k}\lambda_j(z^j-1)
		&u^{\alpha-1}v^{\beta-1}\iint_{\mathbb{R}^2_+}G(z,x,y)e^{-u^\alpha x-v^\beta y}\,\mathrm{d}x\,\mathrm{d}y \\ &\ \ +\sum_{j=1}^{k}\sum_{j'=1}^{k}\lambda_j\lambda_{j'}\sum_{r=1}^{j'}(z^{j+j'-r}-z^{j'-r})\frac{\partial}{\partial\lambda_1}u^{\alpha-1}v^{\beta-1}\iint_{\mathbb{R}^2_+}G(z,x,y)e^{-u^\alpha x-v^\beta y}\,\mathrm{d}x\,\mathrm{d}y\\
		&=u^{\alpha-1}v^{\beta-1}\iint_{\mathbb{R}^2_+}\bigg(\frac{\partial^2}{\partial x\partial y}G(z,x,y)\bigg)e^{-u^\alpha x-v^\beta y}\,\mathrm{d}x\,\mathrm{d}y\\
		&=u^{\alpha-1}v^{\beta-1}u^\alpha\iint_{\mathbb{R}^2_+}e^{-u^\alpha x-v^\beta y}\bigg(\frac{\partial}{\partial y}G(z,x,y)\bigg)\,\mathrm{d}x\,\mathrm{d}y\\
		&=u^{\alpha-1}v^{\beta-1}u^\alpha v^\beta\iint_{\mathbb{R}^2_+}e^{-u^\alpha x-v^\beta y}G(z,x,y)\,\mathrm{d}x\,\mathrm{d}y-u^{\alpha-1}v^{\beta-1}\\
		&=u^\alpha v^\beta\iint_{\mathbb{R}^2_+}e^{-us-vt}G^{\alpha,\beta}(z,s,t)\,\mathrm{d}s\,\mathrm{d}t-u^{\alpha-1}v^{\beta-1},
	\end{align*}
	which coincides with the left-hand side of (\ref{fgprfpgfeqpf1}). In the first step, the interchange of derivative $\partial/\partial\lambda_1$ and the integral is justified because
	\begin{align*}
		\iint_{\mathbb{R}^2_+}\bigg|\frac{\partial}{\partial \lambda_1}G(z,x,y)\bigg|e^{-u^\alpha x-v^\beta y}\,\mathrm{d}x\,\mathrm{d}y&=\iint_{\mathbb{R}^2_+}\bigg|xy(z-1)\exp\bigg(\sum_{j=1}^{k}\lambda_jxy(z^j-1)\bigg)\bigg|e^{-u^\alpha x-v^\beta y}\,\mathrm{d}x\,\mathrm{d}y\\
		&\leq \iint_{\mathbb{R}^2_+}xye^{-u^\alpha x-v^\beta y}\,\mathrm{d}x\,\mathrm{d}y<\infty,\ u>0,\ v>0.
	\end{align*}
	This completes the proof.
\end{proof}
\begin{theorem}\label{thm41}
	The distribution $p^{\alpha,\beta}(n,s,t)=\mathrm{Pr}\{M^{\alpha,\beta}(s,t)=n\}$, $n\ge0$ of FGPRF is given by
	\begin{equation}\label{fgprfdis}
		p^{\alpha,\beta}(n,s,t)=\sum_{\Theta(k,n)}\prod_{j=1}^{k}\frac{(\lambda_js^\alpha t^\beta)^{n_j}}{n_j!}{}_2\Psi_2\Bigg[\begin{matrix}
			(\sum_{j=1}^{k}n_j+1,1),&(\sum_{j=1}^{k}n_j+1,1)\\\\
			(\alpha\sum_{j=1}^{k}n_j+1,\alpha),&(\beta\sum_{j=1}^{k}n_j+1,\beta)
		\end{matrix}\bigg|-\sum_{j=1}^{k}\lambda_js^\alpha t^\beta\Bigg]
	\end{equation}
	for all $(s,t)\in\mathbb{R}^2_+$, where ${}_2\Psi_2$ is the generalized Wright function as defined in (\ref{gwright}).
	
	In particular, the point probabilities in (\ref{fgprfdis}) solve the following system of fractional differential equations:
	\begin{align*}
		\frac{\partial^{\alpha+\beta}}{\partial t^\beta\partial s^\alpha}p^{\alpha,\beta}(n,s,t)=-\sum_{j=1}^{k}\lambda_j(I-B^j)\bigg(1+\sum_{j'=1}^{k}\lambda_{j'}\sum_{r=1}^{j'}B^{j'-r}\frac{\partial}{\partial\lambda_1}\bigg)p^{\alpha,\beta}(n,s,t),\ n\ge0,
	\end{align*}
	with initial condition $p^{\alpha,\beta}(0,0,0)=1$. Here, $B$ denotes the backward shift operator, defined as $B^jp^{\alpha,\beta}(n,s,t)=p^{\alpha,\beta}(n-j,s,t)$. Also, it is assumed that $p^{\alpha,\beta}(n,s,t)=0$ for all $n<0$.
\end{theorem}
\begin{proof}
	For $n\ge0$, we have
	\begin{equation*}
		p^{\alpha,\beta}(n,s,t)=\iint_{\mathbb{R}^2_+}\mathrm{Pr}\{M(x,y)=n\}\mathrm{Pr}\{L^\alpha(s)\in\mathrm{d}x\}\mathrm{Pr}\{L^\beta(t)\in\mathrm{d}y\},
	\end{equation*}
	whose double Laplace transform is given by
	\begin{align}
		\iint_{\mathbb{R}^2_+}e^{-us-vt}&p^{\alpha,\beta}(n,s,t)\,\mathrm{d}s\,\mathrm{d}t\nonumber\\
		&=u^{\alpha-1}v^{\beta-1}\iint_{\mathbb{R}^2_+}\sum_{\Theta(k,n)}\prod_{j=1}^{k}\frac{(\lambda_jxy)^{n_j}}{n_j!}e^{-\lambda_jxy}e^{-u^\alpha x-v^\beta y}\,\mathrm{d}x\,\mathrm{d}y\nonumber\\
		&=v^{\beta-1}\sum_{\Theta(k,n)}\prod_{j=1}^{k}\frac{\lambda_j^{n_j}}{n_j!}\int_{0}^{\infty}y^{\sum_{j=1}^{k}n_j}\frac{u^{\alpha-1}\Gamma(\sum_{j=1}^{k}n_j+1)}{(u^\alpha+\sum_{j=1}^{k}\lambda_jy)^{\sum_{j=1}^{k}n_j+1}}e^{-v^\beta y}\,\mathrm{d}y,\label{thm41pf1}
	\end{align}
	where we have used (\ref{insublap}) to get the first equality. Here, the interchange of integral with Laplace transform and sum is valid due to Fubini's theorem. On taking the inverse Laplace transform on both sides of (\ref{thm41pf1}) with respect to variable $u$ and using the Laplace transform of the three parameter Mittag-Leffler function (see \cite{Kilbas2006})
	\begin{equation}\label{tpml}
		E_{\alpha,\beta}^\gamma(x)\coloneqq\sum_{r=0}^{\infty}\frac{\Gamma(\gamma+r)x^r}{\Gamma(\gamma)\Gamma(\alpha r+\beta)r!},\ \alpha>0, \beta,\gamma, x\in\mathbb{R},
	\end{equation}
	given as 
	\begin{equation}\label{tpmllap}
		\int_{0}^{\infty}e^{-ux}x^{\beta-1}E_{\alpha,\beta}^\gamma(cx)\,\mathrm{d}x=\frac{u^{\alpha\gamma-\beta}}{(u^\alpha-c)^\gamma},\ c\in\mathbb{R},\ u^\alpha>|c|,\ u>0,
	\end{equation}
	we get
	\begin{align*}
		\int_{0}^{\infty}&e^{-vt}p^{\alpha,\beta}(n,s,t)\,\mathrm{d}t\\
		&=v^{\beta-1}\sum_{\Theta(k,n)}\prod_{j=1}^{k}\frac{\lambda_j^{n_j}}{n_j!}\Gamma\bigg(\sum_{j=1}^{k}n_j+1\bigg)s^{\alpha\sum_{j=1}^{k}n_j}\\
		&\hspace{4cm}\cdot\int_{0}^{\infty}y^{\sum_{j=1}^{k}n_j}E_{\alpha, \alpha\sum_{j=1}^{k}n_j+1}^{\sum_{j=1}^{k}n_j+1}\bigg(-\sum_{j=1}^{k}\lambda_jy s^\alpha\bigg)e^{-v^\beta y}\,\mathrm{d}y\\
		&=v^{\beta-1}\sum_{\Theta(k,n)}\prod_{j=1}^{k}\frac{\lambda_j^{n_j}}{n_j!}\Gamma\bigg(\sum_{j=1}^{k}n_j+1\bigg)s^{\alpha\sum_{j=1}^{k}n_j}\sum_{r=0}^{\infty}\frac{\Gamma(\sum_{j=1}^{k}n_j+1+r)\big(-\sum_{j=1}^{k}\lambda_js^\alpha\big)^r}{\Gamma(\sum_{j=1}^{k}n_j+1)\Gamma(\alpha r+\alpha\sum_{j=1}^{k}n_j+1)r!}\\
		&\hspace{4.5cm}\cdot\int_{0}^{\infty}y^{\sum_{j=1}^{k}n_j+r}e^{-v^\beta y}\,\mathrm{d}y\\
		&=\sum_{\Theta(k,n)}\prod_{j=1}^{k}\frac{\lambda_j^{n_j}}{n_j!}s^{\alpha\sum_{j=1}^{k}n_j}\sum_{r=0}^{\infty}\frac{\Gamma(\sum_{j=1}^{k}n_j+1+r)\big(-\sum_{j=1}^{k}\lambda_js^\alpha\big)^r}{\Gamma(\alpha r+\alpha\sum_{j=1}^{k}n_j+1)r!}\frac{\Gamma(\sum_{j=1}^{k}n_j+r+1)}{v^{\beta(\sum_{j=1}^{k}n_j+r)+1}}.
	\end{align*}
	Its inverse Laplace transform with respect to $v$ yields
	\begin{equation*}
		p^{\alpha,\beta}(n,s,t)=\sum_{\Theta(k,n)}\prod_{j=1}^{k}\frac{(\lambda_js^\alpha t^\beta)^{n_j}}{n_j!}\sum_{r=0}^{\infty}\frac{(\Gamma(\sum_{j=1}^{k}n_j+1+r))^2\big(-\sum_{j=1}^{k}\lambda_js^\alpha t^\beta\big)^r}{\Gamma(\alpha r+\alpha\sum_{j=1}^{k}n_j+1)\Gamma(\beta r+\beta\sum_{j=1}^{k}n_j+1)r!},
	\end{equation*}
	where again the interchange of the inverse Laplace transform and integral follows from Fubini's theorem.
	This yields the distribution of FGPRF.
	
	Now, on substituting $G^{\alpha,\beta}(z,s,t)=\sum_{n=0}^{\infty}z^np^{\alpha,\beta}(n,s,t)$ in (\ref{fgprfpgfeq}), we have
	\begin{align}
		&\sum_{n=0}^{\infty}z^n\frac{\partial^{\alpha+\beta}}{\partial t^\beta\partial s^\alpha}p^{\alpha,\beta}(n,s,t)\nonumber\\
		&=\sum_{n=0}^{\infty}\sum_{j=1}^{k}\lambda_j(z^j-1)z^np^{\alpha,\beta}(n,s,t)+\sum_{n=0}^{\infty}\sum_{j=1}^{k}\sum_{j'=1}^{k}\lambda_j\lambda_{j'}\sum_{r=1}^{j'}(z^{j+j'-r}-z^{j'-r})z^n\frac{\partial}{\partial\lambda_1}p^{\alpha,\beta}(n,s,t)\nonumber\\
		&=\sum_{n=0}^{\infty}\sum_{j=1}^{k}\lambda_jz^{n+j}p^{\alpha,\beta}(n,s,t)-\sum_{n=0}^{\infty}\sum_{j=1}^{k}\lambda_jz^np^{\alpha,\beta}(n,s,t)\nonumber\\
		&\ \ +\sum_{j=1}^{k}\sum_{j'=1}^{k}\lambda_j\lambda_{j'}\sum_{r=1}^{j'}\sum_{n=0}^{\infty}z^{n+j+j'-r}\frac{\partial}{\partial\lambda_1}p^{\alpha,\beta}(n,s,t)-\sum_{j=1}^{k}\sum_{j'=1}^{k}\lambda_j\lambda_{j'}\sum_{r=1}^{j'}\sum_{n=0}^{\infty}z^{n+j'-r}\frac{\partial}{\partial\lambda_1}p^{\alpha,\beta}(n,s,t)\nonumber\\
		&=\sum_{j=1}^{k}\lambda_j\sum_{n=j}^{\infty}z^np^{\alpha,\beta}(n-j,s,t)-\sum_{n=0}^{\infty}\sum_{j=1}^{k}\lambda_jz^np^{\alpha,\beta}(n,s,t)\nonumber\\
		&\ \ +\sum_{j=1}^{k}\sum_{j'=1}^{k}\lambda_j\lambda_{j'}\sum_{r=1}^{j'}\sum_{n=j+j'-r}^{\infty}z^{n}\frac{\partial}{\partial\lambda_1}p^{\alpha,\beta}(n-j-j'+r,s,t)\nonumber\\
		&\ \ -\sum_{j=1}^{k}\sum_{j'=1}^{k}\lambda_j\lambda_{j'}\sum_{r=1}^{j'}\sum_{n=j'-r}^{\infty}z^{n}\frac{\partial}{\partial\lambda_1}p^{\alpha,\beta}(n-j'+r,s,t)\nonumber\\
		&=\sum_{n=0}^{\infty}z^n\sum_{j=1}^{k}\bigg(\lambda_jp^{\alpha,\beta}(n-j,s,t)-\lambda_jp^{\alpha,\beta}(n,s,t)+\sum_{j'=1}^{k}\lambda_j\lambda_{j'}\sum_{r=1}^{j'}\frac{\partial}{\partial\lambda_1}p^{\alpha,\beta}(n-j-j'+r,s,t)\nonumber\\
		&\ \ -\sum_{j'=1}^{k}\lambda_j\lambda_{j'}\sum_{r=1}^{j'}\frac{\partial}{\partial\lambda_1}p^{\alpha,\beta}(n-j'+r,s,t)\bigg).\label{thm41pf2}
	\end{align}
	On comparing the coefficients of $z^n$ on both sides of (\ref{thm41pf2}), we get 
	\begin{align*}
		\frac{\partial^{\alpha+\beta}}{\partial t^\beta\partial s^\alpha}&p^{\alpha,\beta}(n,s,t)\\
		&=\sum_{j=1}^{k}\big(\lambda_jp^{\alpha,\beta}(n-j,s,t)-\lambda_jp^{\alpha,\beta}(n,s,t)\big)\\
		&\ \ +\sum_{j=1}^{k}\sum_{j'=1}^{k}\lambda_j\lambda_{j'}\sum_{r=1}^{j'}\frac{\partial}{\partial\lambda_1}\big(p^{\alpha,\beta}(n-j-j'+r,s,t)-p^{\alpha,\beta}(n-j'+r,s,t)\big)\\
		&=-\sum_{j=1}^{k}\lambda_j\big(I-B^j\big)p^{\alpha,\beta}(n,s,t)-\sum_{j=1}^{k}\sum_{j'=1}^{k}\lambda_j\lambda_{j'}\sum_{r=1}^{j'}B^{j'-r}(I-B^{j})\frac{\partial}{\partial\lambda_1}p^{\alpha,\beta}(n,s,t),\ n\ge0.
	\end{align*}
	This completes the proof.
\end{proof}
\begin{remark}
	From Remark \ref{rem21}, for any rectangle $[0,s]\times[0,t]\subset\mathbb{R}^2_+$, the capacity functional of FGPRF is given by 
	\begin{equation*}
		T_M([0,s]\times[0,t])=1-{}_2\Psi_2\Bigg[\begin{matrix}
			(1,1),&(1,1)\\\\
			(1,\alpha),&(1,\beta)
		\end{matrix}\bigg|-\sum_{j=1}^{k}\lambda_js^\alpha t^\beta\Bigg].
	\end{equation*}
\end{remark}
\begin{remark}
	For $k=1$, (\ref{fgprfdis}) reduces to the distribution of a fractional Poisson random field on $\mathbb{R}^2_+$, given as follows:
	\begin{equation*}
		p^{\alpha,\beta}(n,s,t)|_{k=1}=\frac{(\lambda s^\alpha t^\beta)^{n}}{n!}{}_2\Psi_2\Bigg[\begin{matrix}
			(n+1,1),&(n+1,1)\\\\
			(\alpha n+1,\alpha),&(\beta n+1,\beta)
		\end{matrix}\bigg|-\lambda s^\alpha t^\beta\Bigg], n\ge0,
	\end{equation*}
	which coincides with Eq. (53) of \cite{Kataria2024}. Its governing differential equations are 
	\begin{equation*}
		\frac{\partial^{\alpha+\beta}}{\partial t^\beta\partial s^\alpha}p^{\alpha,\beta}(n,s,t)|_{k=1}=-\lambda\bigg(1+\lambda\frac{\partial}{\partial\lambda}\bigg)(I-B)p^{\alpha,\beta}(n,s,t)|_{k=1},\ n\ge0,
	\end{equation*}
	which agrees with Eq. (5.11) of \cite{Vishwakarma2025c}.
\end{remark}

Suppose $\{M^{\alpha,\beta}(s,t),\ (s,t)\in\mathbb{R}^2_+\}\sim FGPRF\{\lambda_j\}_{1\leq j\leq k}$. From Remark \ref{rem41}, we have the following representation of FGPRF: 
\begin{equation*}
	M^{\alpha,\beta}(s,t)\overset{d}{=}\sum_{j=1}^{k}jN_j(L^\alpha(s), L^\beta(t)),\ (s,t)\in\mathbb{R}^2_+,\ 0<\alpha,\beta<1,
\end{equation*}
where $\{L^\alpha(s),\ s\ge0\}$ and $\{L^\beta(t),\ t\ge0\}$ are independent inverse stable subordinators, and conditional on $(L^\alpha(s),L^\beta(t))$, $\{N_j(L^\alpha(s), L^\beta(t)),\ (s,t)\in\mathbb{R}^2_+\}$, $j=1,\dots,k$ are independent fractional Poisson random fields, introduced in \cite{Leonenko2015}. Hence,
\begin{equation*}
	\mathbb{E}M^{\alpha,\beta}(s,t)=\sum_{j=1}^{k}\frac{j\lambda_js^\alpha t^\beta}{\Gamma(\alpha+1)\Gamma(\beta+1)},\ (s,t)\in\mathbb{R}^2_+,
\end{equation*}
where we have used $\mathbb{E}N_j(L^\alpha(s),L^\beta(t))=\lambda_js^\alpha t^\beta/\Gamma(\alpha+1)$ for each $1\leq j\leq k$ (see \cite{Kataria2024}, Eq. (25)).
The variance of FGPRF is given by
{\small\begin{align*}
	\mathbb{V}\mathrm{ar}M^{\alpha,\beta}(s,t)&=\mathbb{E}(\mathbb{V}\mathrm{ar}(M^{\alpha,\beta}(s,t)|(L^\alpha(s),L^\beta(t))))+\mathbb{V}\mathrm{ar}(\mathbb{E}(M^{\alpha,\beta}(s,t)|(L^\alpha(s),L^\beta(t))))\\
	&=\sum_{j=1}^{k}j^2\lambda_j\mathbb{E}L^\alpha(s)L^\beta(t)+\bigg(\sum_{j=1}^{k}j\lambda_j\bigg)^2\mathbb{V}\mathrm{ar}L^\alpha(s)L^\beta(t)\\
	&=\sum_{j=1}^{k}\frac{j^2\lambda_j s^\alpha t^\beta}{\Gamma(\alpha+1)\Gamma(\beta+1)}+\bigg(\sum_{j=1}^{k}j\lambda_js^\alpha t^\beta\bigg)^2\bigg(\frac{4}{\Gamma(2\alpha+1)\Gamma(2\beta+1)}-\frac{1}{\Gamma^2(\alpha+1)\Gamma^2(\beta+1)}\bigg).
\end{align*}}
As the GPRF is a two parameter L\'evy process, for $(0,0)\preceq(s,t)\preceq(s',t')$ in $\mathbb{R}^2_+$, from Theorem 2.2 of \cite{Vishwakarma2025d}, the auto covariance function of FGPRF is given by
	\begin{align*}
		\mathbb{C}\mathrm{ov}(M^{\alpha,\beta}(s,t)&,M^{\alpha,\beta}(s',t'))\\
		&=\mathbb{E}(L^\alpha(s)L^\beta(t))\mathbb{V}\mathrm{ar}M(1,1)+\mathbb{C}\mathrm{ov}(L^\alpha(s)L^\beta(t),L^\alpha(s')L^\beta(t'))(\mathbb{E}M(1,1))^2\\
		&=\mathbb{E}(L^\alpha(s)L^\beta(s'))\sum_{j=1}^{k}j^2\lambda_j+\mathbb{C}\mathrm{ov}(L^\alpha(s)L^\beta(s'),L^\alpha(t)L^\beta(t'))(\sum_{j=1}^{k}j\lambda_j)^2,
	\end{align*}
	where (see \cite{Vishwakarma2025d}, Example 2.3)
	\begin{equation}\label{expinsub}
		\mathbb{E}(L^\alpha(s)L^\beta(t))=\frac{s^\alpha t^{\beta}}{\Gamma(\alpha+1)\Gamma(\beta+1)},\ (s,t)\in\mathbb{R}^2_+
	\end{equation}
	and
	\begin{align}\label{covinsub}
		\mathbb{C}\mathrm{ov}(L^\alpha(s)L^\beta(t),L^\alpha(s')L^\beta(t'))&=\frac{1}{\Gamma(\alpha+1)\Gamma(\alpha)}\int_{0}^{s}((s'-x)^\alpha+(s-x)^\alpha)x^{\alpha-1}\,\mathrm{d}x\nonumber\\
		&\ \ \cdot \frac{1}{\Gamma(\beta+1)\Gamma(\beta)}\int_{0}^{t}((t'-y)^{\beta}+(t-y)^{\beta})y^{\beta-1}\,\mathrm{d}y\nonumber\\
		&\ \ -\frac{(ss')^\alpha(tt')^\beta}{\Gamma^2(\alpha+1)\Gamma^2(\beta+1)},\ (s,t)\preceq(s',t').
	\end{align}

\section{Generalized Skellam point processes}\label{sec5}
The two parameter Skellam random field via independent PRFs and its fractional variants are investigated in \cite{Vishwakarma2025b}. We now define a generalized Skellam point process using independent GPRFs. 
Let $\mathcal{I}\subset\mathbb{R}-\{0\}$ be a finite subset. For a fix $k\ge1$, let $\{M_i(A), A\in\mathcal{A}_d\}\sim\{\lambda_j^{(i)}\}_{1\leq j\leq k}$, $i\in\mathcal{I}$ be independent GPRF. We consider a point process $\{S(A),\ A\in\mathcal{A}_d\}$ defined as
\begin{equation}\label{gspp}
	S(A)\coloneqq \sum_{i\in\mathcal{I}}iM_i(A),\ A\in\mathcal{A}_d.
\end{equation}
For $k=1$ and $d=1$, it reduces to a generalized Skellam process introduced and studied in \cite{Cinque2025}. 
The mgf of (\ref{gspp}) is given by
\begin{equation}\label{gspppgf}
	\mathbb{E}e^{uS(A)}=\prod_{i\in\mathcal{I}}\mathbb{E}e^{uiM_i(A)}=\exp\bigg(\sum_{i\in\mathcal{I}}\sum_{j=1}^{k}\lambda_{j}^{(i)}|A|\big(e^{iju}-1\big)\bigg),\ u\in\mathbb{R}.
\end{equation}
In the next result, we show that (\ref{gspp}) is equal in distribution to a compound Poisson random field as defined in \cite{Vishwakarma2025c}.
\begin{proposition}\label{prop61}
	 Let $(X_1,Y_1),(X_2,Y_2),\dots$ be iid random vectors taking their values in $\mathcal{I}\times\{1,2,\dots,k\}$ such that 
	\begin{equation*}
		\mathrm{Pr}\{X_1=i, Y_1=j\}=\frac{\lambda_j^{(i)}}{\sum_{i\in\mathcal{I}}\sum_{j=1}^{k}\lambda_j^{(i)}},\ \text{for}\ (i,j)\in\mathcal{I}\times\{1,2,\dots,k\},
	\end{equation*}
	where $\lambda_j^{(i)}$'s are positive constants.
	 Let $\{N(A),\ A\in\mathcal{A}_d\}$ be a PRF with rate parameter $\sum_{i\in\mathcal{I}}\sum_{j=1}^{k}\lambda_j^{(i)}$, which is independent of $\{(X_r,Y_r)\}_{r\ge1}$. Then, the point process $\{S(A),\ A\in\mathcal{A}_d\}$ as defined in (\ref{gspp}) satisfies the following equality:
	\begin{equation*}
		S(A)\overset{d}{=}\sum_{r=1}^{N(A)}X_rY_r,\ A\in\mathcal{A}_d.
	\end{equation*}
\end{proposition}
\begin{proof}
	The mgf of $X_1Y_1$ is given by
	\begin{equation*}
		\mathbb{E}e^{uX_1Y_1}=\frac{\sum_{i\in\mathcal{I}}\sum_{j=1}^{k}e^{uij}\lambda_j^{(i)}}{\sum_{i\in\mathcal{I}}\sum_{j=1}^{k}\lambda_j^{(i)}},\ u\in\mathbb{R}.
	\end{equation*}
	Thus,
	\begin{align*}
		\mathbb{E}\exp\bigg(u\sum_{r=1}^{N(A)}X_rY_r\bigg)&=\mathbb{E}\bigg(\frac{\sum_{i\in\mathcal{I}}\sum_{j=1}^{k}e^{uij}\lambda_j^{(i)}}{\sum_{i\in\mathcal{I}}\sum_{j=1}^{k}\lambda_j^{(i)}}\bigg)^{N(A)}\\
		&=\exp\bigg(|A|\sum_{j=1}^{k}\sum_{i\in\mathcal{I}}\lambda_j^{(i)}\bigg(\frac{\sum_{i\in\mathcal{I}}\sum_{j=1}^{k}e^{uij}\lambda_j^{(i)}}{\sum_{i\in\mathcal{I}}\sum_{j=1}^{k}\lambda_j^{(i)}}-1\bigg)\bigg),
	\end{align*}
	which coincides with (\ref{gspppgf}). This completes the proof.
\end{proof}
\begin{remark}
	Let $(X_r,Y_r)$'s be random vectors as defined in Proposition \ref{prop61}. For $k=1$, $Y_r$'s are degenerate random variable with total mass at $1$, and $M_i(A)$'s are PRF with rate parameter $\lambda_1^{(i)}$. Hence 
	$
		S(A)|_{k=1}\overset{d}{=}\sum_{r=1}^{N(A)}X_r,
	$
	where $N(A)$ is the PRF with rate parameter $\sum_{i\in\mathcal{I}}\lambda_1^{(i)}$, that is independent of $\{X_r\}_{r\ge1}$. 
\end{remark}

We now consider the case $\mathcal{I}=\{1,-1\}$. Let $\{M_1(A),\ A\in\mathcal{A}_d\}\sim GPRF\{\lambda_j^{(1)}\}_{1\leq j\leq k}$ and $\{M_2(A),\ A\in\mathcal{A}_d\}\sim\{\lambda_j^{(2)}\}_{1\leq j\leq k}$ be independent GPRFs. We consider a point process $\{\mathcal{S}(A),\ A\in\mathcal{A}_d\}$, 
\begin{equation*}
	\mathcal{S}(A)\coloneqq M_1(A)-M_2(A),\ A\in\mathcal{A}_d.
\end{equation*}
Its mean and variance are given by $\mathbb{E}\mathcal{S}(A)=\sum_{j=1}^{k}j(\lambda_j^{(1)}-\lambda_j^{(2)})|A|$ and $\mathbb{V}\mathrm{ar}\mathcal{S}(A)=\sum_{j=1}^{k}j^2(\lambda_j^{(1)}+\lambda_j^{(2)})|A|$, respectively.  

For each $i=1,2$, from Theorem \ref{thmrep}, it follows that there exist independent PRFs $\{N_j^{(i)}(A),\ A\in\mathcal{A}_d\}$ with respective rate parameters $\lambda_j^{(i)}$, $j=1,\dots,k$, such that $M_i(A)\overset{d}{=}\sum_{j=1}^{k}jN_j^{(i)}(A)$. Hence,
\begin{equation*}
	\mathcal{S}(A)\overset{d}{=}\sum_{j=1}^{k}j(N_j^{(1)}(A)-N_j^{(2)}(A)),\ A\in\mathcal{A}_d,
\end{equation*}
where $N_j^{(1)}(\cdot)$ and $N_j^{(2)}(\cdot)$ are independent for all $j$. The process $\{N_j^{(1)}(A)-N_j^{(2)}(A)\}$ is a Skellam random field introduced and studied in \cite{Vishwakarma2025b}.
Thus, $\mathcal{S}(A)$ is a weighted sum of $k$-many independent Skellam random fields. The one dimensional distribution of $N_j^{(1)}(A)-N_j^{(2)}(A)$ is given as follows (see \cite{Vishwakarma2025b}):
\begin{equation*}
	\mathrm{Pr}\{N_j^{(1)}(A)-N_j^{(2)}(A)=n\}=e^{-(\lambda_j^{(1)}+\lambda_j^{(2)})|A|}(\lambda_j^{(1)}/\lambda_j^{(2)})^{n/2}I_{|n|}(2|A|\sqrt{\lambda_j^{(1)}\lambda_j^{(2)}}),\ n\in\mathbb{Z},\ A\in\mathcal{A}_d,
\end{equation*}
where $\mathbb{Z}$ denotes the set of integers. Suppose $\tilde{\Theta}(k,n)=\{(n_1,\dots,n_k)\in\mathbb{Z}^k:n_1+2n_2+\dots+kn_k=n\}$ for all $n\in\mathbb{Z}$. Then, the distribution of $\mathcal{S}(A)$ is given by
\begin{equation}\label{gsppdist1}
	\mathrm{Pr}\{\mathcal{S}(A)=n\}=e^{-(\Lambda^{(1)}+\Lambda^{(2)})|A|}\sum_{\tilde{\Theta}(k,n)}\prod_{j=1}^{k}(\lambda_j^{(1)}/\lambda_j^{(2)})^{n_j/2}I_{|n_j|}(2|A|\sqrt{\lambda_j^{(1)}\lambda_j^{(2)}}),\ n\in\mathbb{Z},\ A\in\mathcal{A}_d,
\end{equation}
where $\Lambda^{(1)}=\lambda_1^{(1)}+\dots+\lambda_k^{(1)}$ and $\Lambda^{(2)}=\lambda_1^{(2)}+\dots+\lambda_k^{(2)}$. Here, $I_\nu$ is the modified Bessel function of first kind defined as (see \cite{Kilbas2006})
\begin{equation*}
	I_\nu(x)=\sum_{m=0}^{\infty}\frac{(x/2)^{2m+\nu}}{m!\Gamma(m+\nu+1)},\ x\in\mathbb{R},\ \nu>-1.
\end{equation*}

\subsection{Generalized Skellam point process on $\mathbb{R}^2_+$} For $d=2$, the process (\ref{gspp}) reduces to a two parameter process $\{S(s,t),\ (s,t)\in\mathbb{R}^2_+\}$, defined as follows:
\begin{equation}\label{gspp2}
	S(s,t)\coloneqq \sum_{i\in\mathcal{I}}iM_i(s,t),\ (s,t)\in\mathbb{R}^2_+,
\end{equation}
where $\{M_i(s,t),\ (s,t)\in\mathbb{R}^2_+\}\sim GPRF\{\lambda_j^{(i)}\}_{1\leq j\leq k}$, $i\in\mathcal{I}$ are independent GPRFs. We call $\{S(s,t),\ (s,t)\in\mathbb{R}^2_+\}$  the generalized Skellam point process (GSPP) and denote it by $\{S(s,t),\ (s,t)\in\mathbb{R}^2_+\}\sim GSPP\{\lambda_j^{(i)},\ i\in\mathcal{I}\}_{1\leq j\leq k}$. The GSPP is a two parameter L\'evy process with rectangular increments (for definition, see Section \ref{gprf2sec}). Its mgf is given by
\begin{equation}\label{gspppgf2}
	\mathbb{E}e^{uS(s,t)}=\exp\bigg(st\sum_{i\in\mathcal{I}}\sum_{j=1}^{k}\lambda_j^{(i)}(e^{iju}-1)\bigg),\ u\in\mathbb{R}.
\end{equation}
\begin{remark}
	Suppose $\{S(s,t),\ (s,t)\in\mathbb{R}^2_+\}\sim GSPP\{\lambda_j^{(i)},\ i\in\mathcal{I}\}_{1\leq j\leq k}$. Let $\{(X_r,Y_r)\}_{r\ge1}$ be a sequence as in Proposition \ref{prop61}, which is independent of a two parameter PRF $\{N(s,t),\ (s,t)\in\mathbb{R}^2_+\}$ with rate parameter $\sum_{i\in\mathcal{I}}\sum_{j=1}^{k}\lambda_j^{(i)}$. Then, $S(s,t)\overset{d}{=}\sum_{r=1}^{N(s,t)}X_rY_r$ for all $(s,t)\in\mathbb{R}^2_+$.  
\end{remark}
\begin{theorem}
	For $m,l\ge1$ and $n\ge1$, let $p^{(n)}_{m,l,i,j}\in(0,1)$ for all $(i,j)\in\mathcal{I}\times\{1,2,\dots,k\}$ such that $\sum_{i\in\mathcal{I}}\sum_{j=1}^{k}p^{(n)}_{m,l,i,j}<1$. Let $\{(X_{m,l}^{(n)},Y_{m,l}^{(n)})\}_{m,l,n\ge1}$ be a triple indexed sequence of random vectors independent across $m$, $l$ and $n$ such that
	\begin{equation*}
		\mathrm{Pr}\{(X_{m,l}^{(n)},Y_{m,l}^{(n)})=(i,j)\}\coloneqq p_{m,l,i,j}^{(n)}\ \text{for}\ (i,j)\in\mathcal{I}\times\{1,2,\dots,k\}
	\end{equation*}
	and
	\begin{equation*}
		\mathrm{Pr}\{(X_{m,l}^{(n)},Y_{m,l}^{(n)})\in\{(i,0), (0,j), (0,0)\}_{i\in\mathcal{I},1\leq j\leq k}\}=1-\sum_{i\in\mathcal{I}}\sum_{j=1}^{k}p^{(n)}_{m,l,i,j}.
	\end{equation*}
	Let $\{Z^{(n)}(s,t),\ (s,t)\in\mathbb{R}^2_+\}$ be a two parameter process defined as 
	\begin{equation*}
		Z^{(n)}(s,t)\coloneqq \sum_{m=1}^{[ns]}\sum_{l=1}^{[nt]}X_{m,l}^{(n)}Y_{m,l}^{(n)},\ (s,t)\in\mathbb{R}^2_+.
	\end{equation*}
	If 
	\begin{equation}\label{thm61lim2}
		\sum_{m=1}^{[ns]}\sum_{l=1}^{[nt]}p_{m,l,i,j}^{(n)}\rightarrow \lambda_j^{(i)}st,\ (s,t)\in\mathbb{R}^2_+\ \text{and}\ \sup_{1\leq m,\,l\leq n}p_{m,l,i,j}^{(n)}\rightarrow0,
	\end{equation}
	as $n\rightarrow\infty$ for all $(i,j)\in\mathcal{I}\times\{1,2,\dots,k\}$, then for any $(s_1,t_1)\preceq(s_2,t_2)\preceq\dots\preceq(s_r,t_r)$ in $\mathbb{R}^2_+$, we have the following convergence:
	\begin{equation*}
		(Z^{(n)}(s_1,t_1),\dots,Z^{(n)}(s_r,t_r))\overset{d}{\rightarrow}(S(s_1,t_1),\dots,S(s_r,t_r))\ \text{as}\ n\rightarrow\infty.
	\end{equation*}
\end{theorem}
\begin{proof}
For $(s,t)\preceq(s',t')$, the rectangular increment of $Z^{(n)}(s,t)$ is given by
\begin{align*}
	\Delta_{s,t}Z^{(n)}(s',t')&=\sum_{m=1}^{[ns']}\sum_{l=1}^{[nt']}X_{m,l}^{(n)}Y_{m,l}^{(n)}-\sum_{m=1}^{[ns']}\sum_{l=1}^{[nt]}X_{m,l}^{(n)}Y_{m,l}^{(n)}-\sum_{m=1}^{[ns]}\sum_{l=1}^{[nt']}X_{m,l}^{(n)}Y_{m,l}^{(n)}+\sum_{m=1}^{[ns]}\sum_{l=1}^{[nt]}X_{m,l}^{(n)}Y_{m,l}^{(n)}\\
	&=\sum_{m=[ns]+1}^{[ns']}\sum_{l=[nt]}^{[nt']}X_{m,l}^{(n)}Y_{m,l}^{(n)}.
\end{align*}
Thus, $\{Z^{(n)}(s,t),\ (s,t)\in\mathbb{R}^2_+\}$ has independent increments. Now, following the arguments of proof of Theorem 3.1 of \cite{Vishwakarma2025b}, it is enough to show that $\Delta_{s,t}Z^{(n)}(s',t')\overset{d}{\rightarrow}\Delta_{s,t}S(s',t')$ as $n\rightarrow\infty$. The mgf of $\Delta_{s,t}Z^{(n)}(s',t')$ is given by
\begin{align*}
	\mathbb{E}\exp(u\Delta_{s,t}Z^{(n)}(s',t'))&=\prod_{m=[ns]+1}^{[ns']}\prod_{l=[nt]+1}^{[nt']}\mathbb{E}\exp(uX_{m,l}^{(n)}Y_{m,l}^{(n)})\\
	&=\prod_{m=[ns]+1}^{[ns']}\prod_{l=[nt]+1}^{[nt']}\Big(\sum_{i\in\mathcal{I}}\sum_{j=1}^{k}e^{iju}p_{m,l,i,j}^{(n)}+1-\sum_{i\in\mathcal{I}}\sum_{j=1}^{k}p_{m,l,i,j}^{(n)}\Big)\\
	&=\exp\bigg(\sum_{m=[ns]+1}^{[ns']}\sum_{l=[nt]+1}^{[nt']}\ln\bigg(\sum_{i\in\mathcal{I}}\sum_{j=1}^{k}p_{m,l,i,j}^{(n)}(e^{iju}-1)+1\bigg)\bigg)\\
	&\sim \exp\bigg(\sum_{m=[ns]+1}^{[ns']}\sum_{l=[nt]+1}^{[nt']}\sum_{i\in\mathcal{I}}\sum_{j=1}^{k}p_{m,l,i,j}^{(n)}(e^{iju}-1)\bigg),
\end{align*}
where the penultimate step follows from the second condition of (\ref{thm61lim2}), and using the approximation $\ln(x+1)\sim x$ as $x\rightarrow0$. On using the first condition of (\ref{thm61lim2}), we get
\begin{equation*}
	\lim_{n\rightarrow\infty}\mathbb{E}\exp(u\Delta_{s,t}Z^{(n)}(s',t'))=\exp\bigg((s'-s)(t'-t)\sum_{i\in\mathcal{I}}\sum_{j=1}^{k}\lambda_j^{(i)}(e^{iju}-1)\bigg),\ u\in\mathbb{R}.
\end{equation*}
From (\ref{gspppgf2}), the proof follows using the stationary increments property of GSPP.
\end{proof}

Let $\{M_i(s,t),\ (s,t)\in\mathbb{R}^2_+\}\sim GPRF\{\lambda_j^{(i)}\}_{1\leq j\leq k}$, $i=1,2$ be independent GPRFs. We consider a two parameter process $\{S(s,t),\ (s,t)\in\mathbb{R}^2_+\}$,
\begin{equation*}
	S(s,t)\coloneqq M_1(s,t)-M_2(s,t),\ (s,t)\in\mathbb{R}^2_+.
\end{equation*}
We denote it as $\{\mathcal{S}(s,t),\ (s,t)\in\mathbb{R}^2_+\}\sim GSPP\{\lambda_j^{(1)},\ \lambda_j^{(2)}\}_{1\leq j\leq k}$. Its mean and variance are given by $\mathbb{E}\mathcal{S}(s,t)=\sum_{j=1}^{k}j(\lambda_j^{(1)}-\lambda_j^{(2)})st$ and $\mathbb{V}\mathrm{ar}\mathcal{S}(s,t)=\sum_{j=1}^{k}j^2(\lambda_j^{(1)}+\lambda_j^{(2)})st$, respectively. From (\ref{gsppdist1}), its one dimensional distribution is given by
\begin{equation}\label{gsppdist}
	q(n,s,t)=e^{-(\Lambda^{(1)}+\Lambda^{(2)})st}\sum_{\tilde{\Theta}(k,n)}\prod_{j=1}^{k}(\lambda_j^{(1)}/\lambda_j^{(2)})^{n_j/2}I_{|n_j|}(2st\sqrt{\lambda_j^{(1)}\lambda_j^{(2)}}),\ n\in\mathbb{Z},\ (s,t)\in\mathbb{R}^2_+.
\end{equation}
Its pgf is 
\begin{equation}\label{gspppgf3}
	\mathbb{E}z^{\mathcal{S}(s,t)}=\exp\Big(st\sum_{j=1}^{k}(\lambda_j^{(1)}(z^j-1)+\lambda_j^{(2)}(z^{-j}-1))\Big),\ z\in(0,1].
	\end{equation}
	
\subsection{Fractional GSPP on $\mathbb{R}^2_+$} For $0< \alpha,\beta< 1$, let $\{L^\alpha(t),\ t\ge0\}$ and $\{L^\beta(t),\ t\ge0\}$ be inverse stable subordinators. Suppose $\{\mathcal{S}(s,t),\ (s,t)\in\mathbb{R}^2_+\}\sim GSPP\{\lambda_j^{(1)},\ \lambda_j^{(2)}\}_{1\leq j\leq k}$. Let us consider a time-changed two parameter process $\{\mathcal{S}^{\alpha,\beta}(s,t),\ (s,t)\in\mathbb{R}^2_+\}$,
\begin{equation}\label{fgspp1}
	\mathcal{S}^{\alpha,\beta}(s,t)\coloneqq \mathcal{S}(L^\alpha(s),L^\beta(t)),\ (s,t)\in\mathbb{R}^2_+,
\end{equation}
where all the component processes are independent of each other. For $(s,t)\preceq(s',t')$, from Theorem 2.2 of \cite{Vishwakarma2025d}, its mean, variance and auto covariance are 
\begin{align*}
	\mathbb{E}\mathcal{S}^{\alpha,\beta}(s,t)&=\sum_{j=1}^{k}j(\lambda_j^{(1)}-\lambda_j^{(2)})\frac{s^\alpha t^\beta}{\Gamma(\alpha+1)\Gamma(\beta+1)},\\
	\mathbb{V}\mathrm{ar}\mathcal{S}^{\alpha,\beta}(s,t)&=\sum_{j=1}^{k}j^2(\lambda_j^{(1)}+\lambda_j^{(2)})\mathbb{E}(L^\alpha(s)L^\beta(t))+\mathbb{V}\mathrm{ar}L^\alpha(s)L^\beta(t)(\sum_{j=1}^{k}j(\lambda_j^{(1)}-\lambda_j^{(2)}))^2
\end{align*}
and
\begin{align*}
	\mathbb{C}\mathrm{ov}(\mathcal{S}^{\alpha,\beta}(s,t),\mathcal{S}^{\alpha,\beta}(s',t'))&=\sum_{j=1}^{k}j^2(\lambda_j^{(1)}+\lambda_j^{(2)})\mathbb{E}(L^\alpha(s)L^\beta(t))\\
	&\ \ +(\sum_{j=1}^{k}j(\lambda_j^{(1)}-\lambda_j^{(2)}))^2\mathbb{C}\mathrm{ov}(L^\alpha(s)L^\beta(t),L^\alpha(s')L^\beta(t')),
\end{align*}
respectively. Here, $\mathbb{E}(L^\alpha(s)L^\beta(t))$ and $\mathbb{C}\mathrm{ov}(L^\alpha(s)L^\beta(t),L^\alpha(s')L^\beta(t'))$ are as given in (\ref{expinsub}) and (\ref{covinsub}), respectively, and 
\begin{equation*}
\mathbb{V}\mathrm{ar}L^\alpha(s)L^\beta(t)=s^{2\alpha}t^{2\beta}\bigg(\frac{4}{\Gamma(2\alpha+1)\Gamma(2\beta+1)}-\frac{1}{\Gamma(\alpha+1)^2\Gamma(\beta+1)^2}\bigg).
\end{equation*}
\begin{proposition}
	The pgf of (\ref{fgspp1}) is given by
	\begin{equation*}
		\mathbb{E}z^{\mathcal{S}^{\alpha,\beta}(s,t)}={}_2\Psi_2\Bigg[\begin{matrix}
			(1,1),&(1,1)\\\\
			(1,\alpha),&(1,\beta)
		\end{matrix}\bigg|-\phi(z) s^\alpha t^\beta\Bigg],\ z\in(0,1],\ (s,t)\in\mathbb{R}^2_+,
	\end{equation*}
	where $\phi(z)=\sum_{j=1}^{k}(\lambda_j^{(1)}(1-z^j)+\lambda_j^{(2)}(1-z^{-j}))$.
\end{proposition}
\begin{proof}
	From (\ref{fgspp1}), we have
	\begin{equation*}
		\mathbb{E}z^{\mathcal{S}^{\alpha,\beta}(s,t)}=\iint_{\mathbb{R}^2_+}\mathbb{E}z^{\mathcal{S}(x,y)}\mathrm{Pr}\{L^\alpha(s)\in\mathrm{d}x,\,L^\beta(t)\in\mathrm{d}y\},\ z\in(0,1].
	\end{equation*}
	On using (\ref{insublap}), its double Laplace transform is given by
	\begin{align*}
		\iint_{\mathbb{R}^2_+}e^{-us-vt}\mathbb{E}z^{\mathcal{S}^{\alpha,\beta}(s,t)}\,\mathrm{d}s\,\mathrm{d}t&=u^{\alpha-1}v^{\beta-1}\iint_{\mathbb{R}^2_+}e^{-xy\phi(z)}e^{-u^\alpha x}e^{-v^\beta y}\,\mathrm{d}x\,\mathrm{d}y\\
		&=u^{\alpha-1}v^{\beta-1}\int_{0}^{\infty}\frac{e^{-v^\beta y}}{u^\alpha+y\phi(z)}\,\mathrm{d}y,
	\end{align*}
	where we have used (\ref{gspppgf3}). Its inversion with respect to $u$ gives
	\begin{align*}
		\int_{0}^{\infty}e^{-vt}\mathbb{E}z^{\mathcal{S}^{\alpha,\beta}(s,t)}\,\mathrm{d}t&=v^{\beta-1}\int_{0}^{\infty}E_{\alpha,1}(-y\phi(z)s^\alpha)e^{-v^\beta y}\,\mathrm{d}y\\
		&=v^{\beta-1}\sum_{m=0}^{\infty}\frac{(-\phi(z)s^\alpha)^m}{\Gamma(m\alpha+1)}\int_{0}^{\infty}y^me^{-v^\beta y}\,\mathrm{d}y\\
		&=\sum_{m=0}^{\infty}\frac{(-\phi(z)s^\alpha)^m\Gamma(m+1)}{\Gamma(m\alpha+1)v^{\beta m+1}},
	\end{align*} 
	whose inverse Laplace transform yields
	\begin{equation*}
		\mathbb{E}z^{\mathcal{S}^{\alpha,\beta}(s,t)}=\sum_{m=0}^{\infty}\frac{(-\phi(z)s^\alpha t^\beta)^m\Gamma(m+1)}{\Gamma(m\alpha+1)\Gamma(\beta m+1)}.
	\end{equation*}
	This completes the proof.
\end{proof}

The following result gives the state probabilities of (\ref{fgspp1}). On using (\ref{gsppdist}), its proof follows similar lines to that of Theorem \ref{thm41}. Hence, we omit it.
\begin{theorem}
	The state probabilities of (\ref{fgspp1}) is given by
	\begin{align}\label{fgsppdist}
		\mathrm{Pr}&\{\mathcal{S}^{\alpha,\beta}(s,t)=n\}\nonumber\\
		&=\sum_{\tilde{\Theta}(k,n)}\prod_{j=1}^{k}\bigg(\frac{\lambda_j^{(1)}}{\lambda_j^{(2)}}\bigg)^{n_j/2}\sum_{m_1,\dots,m_k=0}^{\infty}\prod_{j=1}^{k}\frac{(\sqrt{\lambda_j^{(1)}\lambda_j^{(2)}}s^\alpha t^\beta)^{2m_j+|n_j|}}{(|n_j|+m_j)!m_j!}\nonumber\\
		&\ \ \cdot{}_2\Psi_2\Bigg[\begin{matrix}
			(\sum_{j=1}^{k}(2m_j+|n_j|)+1,1)&(\sum_{j=1}^{k}(2m_j+|n_j|)+1,1)\\\\
			(\sum_{j=1}^{k}(2m_j+|n_j|)\alpha+1,\alpha)&(\sum_{j=1}^{k}(2m_j+|n_j|)\beta+1,\beta)
		\end{matrix}\bigg|-(\Lambda^{(1)}+\Lambda^{(2)})s^\alpha t^\beta\Bigg],\ n\in\mathbb{Z}.
	\end{align}
\end{theorem}
\begin{remark}
	For $k=1$, the state probabilities (\ref{fgsppdist}) reduces to 
	\begin{align*}
		\bigg(\frac{\lambda_1^{(1)}}{\lambda_1^{(2)}}\bigg)^{n/2}&\sum_{m=0}^{\infty}\frac{(\sqrt{\lambda_1^{(1)}\lambda_1^{(2)}}s^\alpha t^\beta)^{2m+|n|}}{(|n|+m)!m!}\nonumber\\
		&\ \ \cdot{}_2\Psi_2\Bigg[\begin{matrix}
			(2m+|n|+1,1)&(2m+|n|)+1,1)\\\\
			((2m+|n|)\alpha+1,\alpha)&((2m+|n|)\beta+1,\beta)
		\end{matrix}\bigg|-(\lambda^{(1)}_1+\lambda^{(2)}_1)s^\alpha t^\beta\Bigg],\ n\in\mathbb{Z},
	\end{align*}
	which equal the point probabilities of a fractional random field, given in Eq. (4.8) of \cite{Vishwakarma2025b}.
\end{remark}
\vspace{0.4cm}

\paragraph{\textbf{Acknowledgement}} This work is partially supported by the National Post Doctoral Fellowship, PDF/2025/000076, from Anusandhan National Research Foundation, Govt. of India.

\end{document}